\documentclass{amsart}
\usepackage[utf8]{inputenc}
\usepackage[T1]{fontenc}
\usepackage{enumerate}
\usepackage{xcolor}
\usepackage{imakeidx}
\usepackage{amsfonts}
\usepackage{amssymb}
\usepackage{amsmath}
\usepackage{amsthm}
\usepackage{tikz-cd}
\usepackage{float}
\usepackage{graphicx}
\usepackage{tikz}
\usepackage{thmtools}
\usepackage{thm-restate}
\usepackage{mathrsfs}
\usepackage{enumitem}
\usepackage{mathtools}
\usepackage{makecell} 
\usepackage{wrapfig}
\usepackage{verbatim}

\usepackage{hyperref}
\usepackage[backend=biber, style=alphabetic, sorting=nyt, maxcitenames=10, minbibnames=10, maxbibnames=10]{biblatex}
\newtheorem{theorem}{Theorem}[subsection]
\newtheorem{corollary}[theorem]{Corollary}
\newtheorem{proposition}[theorem]{Proposition}
\newtheorem{conjecture}[theorem]{Conjecture}

\newtheorem{lemma}[theorem]{Lemma}
\theoremstyle{definition}
\newtheorem{definition}[theorem]{Definition}
\newtheorem{statement}[theorem]{Statement}

\newtheorem*{notation*}{Notation}

\newtheorem{example}[theorem]{Example}

\newtheorem{remark}[theorem]{Remark}

\newtheorem{propdef}[theorem]{Proposition and definition}

\newtheorem*{question*}{Question I}
\newtheorem*{question'}{Question II}
\newtheorem*{question''}{Question III}

\newtheorem*{theorem*}{Theorem}

\newcommand{\QQ}{{\mathbb Q}}

\newcommand{\bbA}{\mathbb{A}}

\newcommand{\bbC}{\mathbb{C}}
\newcommand{\bbH}{\mathbb{H}}

\newcommand{\bbP}{\mathbb{P}}
\newcommand{\bbQ}{\mathbb{Q}}
\newcommand{\bbR}{\mathbb{R}}
\newcommand{\bbZ}{\mathbb{Z}}

\newcommand{\bbG}{\mathbb{G}}
\newcommand{\bbT}{\mathbb{T}}

\newcommand{\sA}{\mathcal{A}}
\newcommand{\sB}{\mathcal{B}}

\newcommand{\sD}{\mathcal{D}}
\newcommand{\sE}{\mathcal{E}}

\newcommand{\sG}{\mathcal{G}}
\newcommand{\sH}{\mathcal{H}}

\newcommand{\sL}{\mathcal{L}}
\newcommand{\sM}{\mathcal{M}}

\newcommand{\sO}{\mathcal{O}}

\newcommand{\sQ}{\mathcal{Q}}

\newcommand{\sS}{\mathcal{S}}
\newcommand{\sT}{\mathcal{T}}
\newcommand{\sU}{\mathcal{U}}

\newcommand{\sX}{\mathcal{X}}
\newcommand{\sY}{\mathcal{Y}}

\DeclareMathOperator{\Hom}{Hom}
\DeclareMathOperator{\Frac}{Frac}
\DeclareMathOperator{\Spec}{Spec}

\DeclareMathOperator{\Sym}{Sym}

\DeclareMathOperator{\Ext}{Ext}

\DeclareMathOperator{\sym}{sym}

\DeclareMathOperator{\GVF}{GVF}

\DeclareMathOperator{\an}{an}

\DeclareMathOperator{\cdiv}{div} 

\DeclareMathOperator{\Mdiv}{MDiv}

\DeclareMathOperator{\Ldiv}{LDiv}
\DeclareMathOperator{\LdivQ}{LDiv_{\QQ}}

\DeclareMathOperator{\Pic}{Pic}
\DeclareMathOperator{\Div}{Div} 
\DeclareMathOperator{\ord}{ord} 
\DeclareMathOperator{\height}{ht}

\DeclareMathOperator{\id}{id}

\DeclareMathOperator{\Supp}{Supp}

\DeclareMathOperator{\tr}{tr}
\DeclareMathOperator{\Stab}{Stab}
\DeclareMathOperator{\codim}{codim}
\DeclareMathOperator{\Alb}{Alb}

\newcommand{\T}{\bbT} 
\newcommand{\abmoduli}{\bbA}
\newcommand{\tautlb}{\mathfrak L}
\newcommand{\semiabmoduli}{\bbG}
\newcommand{\univabvar}{\mathfrak A}
\newcommand{\univsemiab}{\sG}
\newcommand{\univsubvarofab}{\bbH}

\newcommand{\Tbun}{\sT} 

\newcommand{\aTbun}{\widehat{\Tbun}}
\newcommand{\toricbun}{\sX} 
\newcommand{\toricvar}{X} 

\newcommand{\arho}{\widehat{\rho}}

\newcommand{\essmin}{\zeta_{\textrm{ess}}}
\newcommand{\absmin}{\zeta_{\textrm{abs}}}

\DeclareSymbolFont{yhlargesymbols}{OMX}{yhex}{m}{n} \DeclareMathAccent{\yhwidehat}{\mathord}{yhlargesymbols}{"62}

\newcommand{\Qbar}{\ov{\mathbb{Q}}}

\newcommand{\Set}{\mathbf{Set}}

\newcommand{\NTht}{\widehat{h}}
\newcommand{\Falht}{h_{\textrm{Fal}}}

\newcommand{\Hilbert}{\underline{\mathbf{Hilb}}}
\newcommand{\resHilbert}{\Hilbert^\circ}

\newcommand{\Field}{K}

\newcommand{\adiv}{\widehat{\cdiv}}

\newcommand{\aDivR}{\widehat{\Div}_\bbR}

\newcommand{\aDivQ}{\widehat{\Div}_\bbQ}

\newcommand{\aPic}{\widehat{\Pic}}

\newcommand{\adeg}{\widehat{\deg}}

\newcommand{\GVFan}[2][]{\ensuremath{#2_{\GVF\ifx&#1& \else ,#1 \fi}}}
\newcommand{\GVFansca}[2][]{\ensuremath{#2^{\geq 0}_{\GVF\ifx&#1& \else ,#1 \fi}}}

\newcommand{\ov}{\overline}

\newcommand{\blank}{{-}}

\newcommand{\ktbar}{\ov{k(t)}}

\title[New gap principle for semiabelian varieties]{New gap principle for semiabelian varieties using globally valued fields}
\author{Nuno Hultberg}
\address{Institute for Theoretical Sciences, Westlake University, No. 600 Dunyu Road, Sandun town, Xihu district, Hangzhou, Zhejiang Province, 310030 China}
\email{\url{nuno.hultberg@outlook.com}}
\date{}

\begin{document}

\maketitle

\begin{abstract}
    Hrushovski observed that the new gap principle of Gao-Ge-K\"uhne is essentially equivalent to the Bogomolov conjecture over arbitrary globally valued fields of characteristic $0$. Building on this observation, we prove a new gap principle for semiabelian varieties by reducing the Bogomolov conjecture for semiabelian varieties to the Bogomolov conjecture for abelian varieties over arbitrary GVFs. This reduction remains valid in positive characteristic; however, the corresponding Bogomolov conjecture for abelian varieties is not yet known in that setting. We prove an unconditional new gap principle in positive characteristic for semiabelian varieties whose abelian quotient is an elliptic curve.
\end{abstract}

\tableofcontents

\section{Introduction}

Methods to prove geometric and relative versions of the Bogomolov conjecture have been closely related to uniform versions of the Bogomolov conjecture such as the new gap principle, see, for example, \cite{dgh_relativeimpliesuniform},\cite{cantat_gao_habegger_geometric_bogomolov} and \cite{mavraki_schmidt_dyn_Bogomolov_families}. The theory of globally valued fields and their ultraproducts allows one to make this relationship precise. This idea was put forward by Hrushovski in his talk for the conference "The Mordell conjecture 100 years later" in July 2024. In summary, the new gap principle for abelian varieties over $\Qbar$ is essentially equivalent to the Bogomolov conjecture for all globally valued fields (GVFs) of characteristic $0$.

The new gap principle for abelian varieties was proven in \cite{gao_ge_kuehne_uniform_mordell_lang} and used to deduce uniform Mordell-Lang in characteristic $0$. The main result of this article is the extension of the new gap principle to semiabelian varieties. This result has also been obtained independently in work in progress by other authors using different methods who use it to deduce the uniform Mordell-Lang conjecture for semiabelian varieties in characteristic $0$. Our approach has the advantage that it uses results on abelian varieties as a black box. This allows us to obtain partial results in positive characteristic.

Globally valued fields, defined by Ben Yaacov and Hrushovski, are fields together with the data allowing for a theory of heights over them. As such, they are closely related to proper adelic curves, introduced by Chen and Moriwaki, fields with a measure space of absolute values satisfying the product formula. The Bogomolov conjecture has been proven for adelic curves with a positive measure of archimedean absolute values by equidistribution methods in \cite{Adelic_curves_4}. Our first result is an extension of this theorem to all globally valued fields of characteristic $0$.

Let $K$ be a globally valued field whose underlying field is algebraically closed and $A$ an abelian variety over $K$. Denote by $\NTht$ the canonical height on $A$. If $K$ is non-archimedean, we denote by $k$ the subfield of elements of height $0$. Let $(A^{K/k}, \tr)$ denote the $K/k$-Chow trace of $A$. The trace is the universal abelian variety defined over $k$, equipped with a morphism $A^{K/k}\otimes K \to A$.

A closed integral subvariety of $A$ is called special if it is of the form
\[
    \tr(X\otimes K) + V + a,
\]
where $X$ is a subvariety of $A^{K/k}$, $V$ is an abelian subvariety of $A$ and $a\in A(K)$ is a point satisfying $\widehat{h}(a) = 0$. 

\textbf{Bogomolov Conjecture for abelian varieties (BC).} \textit{Any subvariety $X \subset A$ such that, for every $\epsilon > 0$, the set $\{x\in X(K)\ |\ \NTht(x) < \epsilon\}$ is Zariski dense is special. This is moreover equivalent to $\NTht(X) = 0$.}

Here $\NTht(X) = 0$ denotes the normalized N\'eron-Tate height of the variety $X.$

\begin{restatable}{rthm}{bogomolovabvarcharzero}
\label{thm:bogomolov_characteristic_0}
    $\mathrm{(BC)}$ is true for every abelian variety over a GVF of characteristic $0$.
\end{restatable}

This result provides a partial answer to \cite[Remark 5.3]{Adelic_curves_4} as any adelic curve gives rise to a globally valued field and the statement depends only on the GVF structure one obtains. We use this to rederive the new gap principle of Gao-Ge-K\"uhne and a geometric version of the new gap principle in characteristic $0$. It is not an independent proof as our proof of the geometric Bogomolov conjecture for non-archimedean GVFs uses their results on potential non-degeneracy.

The perspective of working over globally valued fields is helpful for our work on semiabelian varieties. Inspired by the proof of the geometric Bogomolov conjecture for semiabelian varieties of Luo and Yu, \cite{luo2025geometricbogomolovconjecturesemiabelian}, we reduce the Bogomolov conjecture for semiabelian varieties to the Bogomolov conjecture for abelian varieties in the setting of globally valued fields.

Let us give a short introduction to the canonical height on a semiabelian variety. We refer to \S \ref{sec:bogomolov} for details. Let $G$ be a semiabelian variety over $K$ with abelian quotient $\pi: G \to A$. Fix a canonically metrized ample symmetric line bundle $\ov{\sM}$ on $A$ with underlying line bundle $M$ and an integral polytope $\Delta \subset \bbR^n$ containing the origin in its interior. We define the canonical height $\widehat{h} = h_{\pi^*\ov{\sM}} + h_{\arho(\Delta)}$ on $G$ as the sum of an abelian and a toric contribution. Here $\arho(\Delta)$ denotes an adelic Cartier divisor on a compactification of $G$. Its underlying Cartier divisor is denoted $\rho(\Delta)$. It is obtained by equivariantly extending a canonically metrized toric divisor on a compactification of the torus part and admits an isometry $[n]^*\arho(\Delta) \cong \arho(\Delta)^{\otimes n}$. Explicitly, we may take $\Delta = [-1,1]^t$. Then, the Cartier divisor $\rho(\Delta)$ is defined on the compactification $\ov{G} = (G \times (\bbP^1)^t)/(\bbG_m^t)$, where $x \in \bbG_m^t$ acts as $(x,x^{-1})$, as $\sum^{t-1}_{i=0} (G \times (\bbP^1)^i\times([0] + [\infty])\times (\bbP^1)^{t-i-1})/(\bbG_m^t)$. $\arho(\Delta)$ is the unique adelic Cartier divisor extending $\rho(\Delta)$ such that the identification $[n]^*\arho(\Delta) \cong \arho(\Delta)^{\otimes n}$ is an isometry.

Let $(G^{K/k}, \tr)$ denote the $K/k$-Chow trace of $G$. The trace is the universal semiabelian variety defined over $k$ with a map $G^{K/k}\otimes K \to G$. Its existence was recently established in \cite[Corollary 2.4.4 (2)]{liu2024chowtrace1motiveslangneron}.\footnote{A treatment avoiding the existence of the Chow trace is possible, cf.\ \cite[\S 3.2]{luo2025geometricbogomolovconjecturesemiabelian}, but its existence simplifies the statements.} A closed integral subvariety $X\subset G$ is called special if 
\[
    \widetilde X:=X/\mathrm{Stab}(X)=\tr(X_0\otimes K)  + g,
\]
where $X_0$ is a subvariety of $\widetilde G^{K/k}$ and $g\in \widetilde G(F)$ satisfying $\widehat{h}(g) = 0$, where $\widetilde G:=G/\mathrm{Stab}(X)$ and $F$ is a GVF extension of $K$. Let $\pi:G \to A$ denote the projection to its abelian quotient.

\textbf{Bogomolov Conjecture for semiabelian varieties (BCSA).} \textit{Suppose there exists a GVF extension $F$ such that for every $\epsilon > 0$ the set $\{x\in X(F)\ |\ \NTht(x) < \epsilon\}$ is Zariski dense. Then, $X$ is special.}

Our main result is a reduction to the case of abelian varieties.

\begin{restatable}{rthm}{bcsa}
\label{thm:bcsa}
    Let $G$ be a semiabelian variety with abelian quotient $A$. If $\mathrm{(BC)}$ holds for $A$, then $\mathrm{(BCSA)}$ holds for $G$.

    In particular, $\mathrm{(BCSA)}$ holds over all GVFs of characteristic $0$. Furthermore, $\mathrm{(BCSA)}$ holds over GVFs of any characteristic for semiabelian varieties $G$ whose abelian quotient is an elliptic curve.
\end{restatable}

As a consequence, we deduce a new gap principle for semiabelian varieties. Let $G$ be a semiabelian variety over a GVF $K$ together with an identification of its torus part $\bbT \cong \bbG_m^t$ and an ample symmetric line bundle $M$ on its abelian quotient $A$. Due to the Barsotti-Weil formula, the semiabelian variety $G$ gives rise to points $Q_1, \dots, Q_n \in A^\vee(K)$. We associate to $G$ the height $h = \Falht + \NTht_{\textrm{NT}}(Q_1) + \dots+\NTht_{\textrm{NT}}(Q_t)$, where $\NTht_{\textrm{NT}}$ denotes the N\'eron-Tate height on $A^\vee$. Fix an integral polytope $\Delta \subset \bbR^n$ containing the origin in its interior. The above data determine a choice of canonical height $\NTht = h_{\pi^*\ov{\sM} + \arho(\Delta)}$ on $G$.

\textbf{New gap principle for semiabelian varieties.} \textit{There exist positive constants $c_1 = c_1(\dim G, \deg_{\pi^*M + \rho(\Delta)} X)$ and $c_2 = c_2(\dim G, \deg_{\pi^*M + \rho(\Delta)} X)$ such that
    \[
        \left\{P \in X(K)\ |\ \NTht(P) < c_1 \max\{\height(2),h(G)\}\right\}
    \]
is contained in a proper Zariski closed subset $X' \subsetneq X$ with $\deg_{\pi^*M + \rho(\Delta)}(X') < c_2$ for all semiabelian varieties $G$ over a GVF $K$ and closed subvarieties $X \subset G$ with finite stabilizer such that $X - X$ generates $G$.}

Hrushovski's observation for abelian varieties is equally valid over semiabelian varieties.

\begin{restatable}{rprop}{bogimpliesngp}\label{prop:bogimpliesngp}
    The Bogomolov conjecture for semiabelian varieties over all GVFs implies the new gap principle.
\end{restatable}

Our partial results on the Bogomolov conjecture allow us to obtain several conclusions.

\begin{restatable}[New gap principle over $\Qbar$]{rthm}{ngpqbar}\label{thm:ngpqbar}
    There exist positive constants $c_1 = c_1(\dim G, \deg_{\pi^*M + \rho(\Delta)} X)$ and $c_2 = c_2(\dim G, \deg_{\pi^*M + \rho(\Delta)} X)$ with the property
    \[
        \left\{P \in X(\Qbar)\ |\ \NTht(P) \leq c_1 \max\{1,h(G)\}\right\}
    \]
    is contained in some Zariski closed $X' \subsetneq X$ with $\deg_{\pi^*M + \rho(\Delta)}(X') < c_2$ for all semiabelian varieties $G$ over $\Qbar$ and closed subvarieties $X \subset G$ with finite stabilizer such that $X - X$ generates $G$.
\end{restatable}

For non-archimedean GVFs we can also formulate a new gap principle.

\begin{restatable}[New gap principle in big characteristic]{rthm}{ngpbigchar}\label{thm:ngpbigchar}
    Let $n$ and $d$ be natural numbers. Then there exist positive constants $c_1, c_2$ and $c_3$ depending only on $n$ and $d$ with the property
    \[
        \left\{P \in X(K)\ |\ \NTht(P) \leq c_1 h(G)\right\}
    \]
    is contained in some Zariski closed $X' \subsetneq X$ with $\deg_{\pi^*M + \rho(\Delta)}(X') < c_2$ for all semiabelian varieties $G$ of dimension $n$ over a GVF $K$ of characteristic $0$ or positive characteristic $>c_3$ and closed subvarieties $X \subset G$ of degree $\deg_{\pi^*M + \rho(\Delta)} X = d$ with finite stabilizer such that $X - X$ generates $G$.
\end{restatable}

In particular, the new gap principle for semiabelian varieties holds in characteristic $0$. It is also possible to obtain results in positive characteristic provided one can prove the corresponding Bogomolov conjecture.

\begin{restatable}{rthm}{ngpellipticquotient}\label{thm:ngpellipticquotient}
    The new gap principle holds in all characteristics for semiabelian varieties whose abelian quotient is an elliptic curve with constants independent of the characteristic.
\end{restatable}

The proof of $\mathrm{BCSA}$ is inspired by the proof of the geometric Bogomolov conjecture for semiabelian varieties in \cite{luo2025geometricbogomolovconjecturesemiabelian}. As in their paper we first prove the Bogomolov conjecture for quasi-split semiabelian varieties, products of an abelian variety and an isotrivial variety, and then reduce to this case. However, the individual steps present several added difficulties the most remarkable being the existence of points of height $0$ that are not torsion. We overcome these difficulties by reducing to the case that the abelian quotient is an elliptic curve assuming the Bogomolov conjecture for the abelian quotient and then using properties of the equidistribution measure on the abelian quotient.

The reduction of the new gap principle to the Bogomolov conjecture relies on degenerations along ultrafilters or equivalently compactness in continuous model theory. This strategy was proposed by Hrushovski in the global setting. A similar approach of degenerating along ultrafilters has been used in the local setting by Yusheng Luo, cf.\ \cite{luo_trees_lengths-spectra} and \cite{luo_limits_of_rational_maps}. It was later formalized in the setting of Berkovich spaces by Favre-Gong, cf.\ \cite{favre_gong_dynamical_degenerations}. In recent work of Yap, \cite{yap2025numbersmallpointsrational}, this has been used to prove a uniformity result in arithmetic dynamics.

The non-archimedean new gap principle in characteristic $0$ for abelian varieties has been proven independently by Mavraki and Yap in \cite{mavraki2025quantitativedynamicalzhangfundamental}. It is possible to deduce such a result formally from the new gap principle over $\Qbar$ in \cite{gao_ge_kuehne_uniform_mordell_lang} using compactness arguments in continuous model theory. Both our works use only the results on potential non-degeneracy found in \cite{gao_ge_kuehne_uniform_mordell_lang}. Potential non-degeneracy may be rephrased as the statement that a subvariety $X$ of a traceless abelian variety over a geometric GVF of characteristic $0$ is special if and only if $\NTht(X) = 0$, cf.\ Lemma \ref{lemm:traceless_abelian}. In the present article the deduction of the non-archimedean new gap principle relies on the proof that this is furthermore equivalent to the existence of a Zariski dense set of small points. In \cite{mavraki2025quantitativedynamicalzhangfundamental}, the deduction relies on a uniform Zhang inequality for dynamical systems. 

The two approaches are closely related. Our results rely on Zhang inequalities over GVFs. For this, we introduce a better-behaved notion of successive minima. This both rectifies the failure of Zhang’s inequalities observed in \cite{guo2025nefconesuccessiveminima} and provides a conceptual explanation of that failure. For our alternative definition of successive minima the Zhang inequalities hold. The height of a variety $X$ is normalized as $h_{\sL}(X) = \frac{(\ov{\sL})^{\dim X + 1}}{(\dim X + 1) (\sL)^{\dim X}}$.

\begin{restatable}[Proposition \ref{prop:lower_zhang_ineq} and \ref{prop:zhang_inequality}]{rprop}{gvfzhangineq}
    Let $\sL$ be a geometrically ample semipositive adelic line bundle on a projective variety $X$ over a GVF $K$. Then,
    \[
        \frac{d}{d+1}\absmin(\sL)+ \frac{1}{d+1}\essmin(\sL) \leq h_{\sL}(X)  \leq \essmin(X).
    \]
\end{restatable}

We can deduce a uniform version of the Zhang inequality from the Zhang inequality for GVFs that differs substantially from \cite[Theorem 1.1]{mavraki2025quantitativedynamicalzhangfundamental}.

\begin{restatable}[Corollary \ref{cor:uniformzhang}]{rthm}{uniformzhangineq}
    Let $\sX \to S$ be a family of projective varieties over $\Qbar$ and let $\ov{\sL}$ be an adelic line bundle on $\sX$ that is semipositive on fibers. Assume furthermore that the geometric Deligne pairing $\langle \sL, \dots,\sL\rangle$ is big on $S$. Then, for any $\epsilon > 0$ there exist constants $N$ and $C$ such that the set 
    \[
        \{x\in \sX(s)(\Qbar)\ |\ h_{\sL}(x) < (1-\epsilon) h_{\sL}(\sX(s))\}
    \]
    is contained in a proper subvariety of degree $\leq N$ and height $< C$ for all $s \in S(\Qbar)$ outside a proper closed subset of $S$.
\end{restatable}

It lacks the quantitative control of \cite[Theorem 1.1]{mavraki2025quantitativedynamicalzhangfundamental}, as its proof relies on a compactness argument, and we require the bigness of a certain Deligne pairing. We note that this bigness does not affect the application to uniform Bogomolov conjectures as this is used as input. On the other hand our statement applies outside the dynamical context to any sufficiently positive line bundles. Furthermore, our approach allows to choose $\epsilon$ arbitrarily. The fact that bigness of the Deligne pairing isn't required in \cite{mavraki2025quantitativedynamicalzhangfundamental} suggests that a version of Zhang's inequality also holds for higher rank analogues of globally valued fields.

Beyond the immediate results of the paper, we believe that the proofs provide a new perspective on uniformity statements in Arakelov geometry and the theory of adelic curves and GVFs. Indeed, most results for adelic curves of characteristic $0$ have a uniform statement for number fields as their consequence and vice versa. This is not only a theoretical observation. The theory of adelic curves by Chen and Moriwaki is developed sufficiently far to apply this observation in practice. This, in turn, motivates further study of adelic curves and GVFs. On the other hand the comparison of different perspectives on GVFs allows one to prove results on adelic curves by spreading out to a family and using the adelic line bundles of Yuan-Zhang. Moreover, the existential closedness of $\Qbar$ and $\ktbar$ makes it possible to extend some results in classical Arakelov geometry to arbitrary adelic curves/GVFs.

We emphasize that in our application it is not sufficient to only prove results for polarized adelic curves. The Bogomolov conjecture is needed for arbitrary adelic curve structures. In other applications, it may be necessary to prove results for a restricted class of adelic curves. For instance the methods of \cite{yap2025numbersmallpointsrational} allow to prove a Bogomolov property for all dynamical systems on $\bbP^1$ over an adelic curve with a place of bad reduction that has positive measure with a bound depending on how much this place contributes to the height of the minimal resultant. This can be used to deduce the uniform statement in loc.cit.\ following methods similar to the present paper.

\subsection{Organization of the paper}

Section \ref{sec:gvf} contains an introduction to globally valued fields. This is partially based on \cite{continuity_of_heights} and \cite{basics_of_gvfs}, but also contains new results like the Zhang inequalities in sections \S \ref{sec:zhang_inequalities}. In section \ref{sec:bogomolov}, we state the Bogomolov conjecture for semiabelian varieties over globally valued fields and discuss the equivalence of a statement using intersection numbers and another using the essential minimum. In section \ref{sec:bogomolov_gvf}, we relate this to the new gap principle and deduce the new gap principle from the Bogomolov conjecture for GVFs by proving Proposition \ref{prop:bogimpliesngp}. We specialize to deduce Theorem \ref{thm:ngpqbar}, \ref{thm:ngpbigchar} and \ref{thm:ngpellipticquotient} from the cases of the Bogomolov conjecture we will later prove. Section \ref{sec:abelian_varieties_and_tori} is the heart of the paper. We first deduce instances of the Bogomolov conjecture over GVFs for abelian varieties and tori that are easily deduced from known results. We finally proceed to prove the reduction of the Bogomolov conjecture for semiabelian varieties to the case of abelian varieties, Theorem \ref{thm:bcsa}, and the Bogomolov conjecture for abelian varieties in characteristic $0$, Theorem \ref{thm:bogomolov_characteristic_0}.

\addtocontents{toc}{\protect\setcounter{tocdepth}{0}}
\subsection*{Acknowledgements}

I thank Long Liu for providing me with a reference on the Chow trace. I thank Wenbin Luo, Michal Szachniewicz, Jit Wu Yap, Huayi Chen and Paolo Dolce for conversations related to this article.
\addtocontents{toc}{\protect\setcounter{tocdepth}{2}}

\section{\label{sec:gvf}Preliminaries on globally valued fields}

The analogy between function fields and number fields as exemplified by the product formula has been very fruitful in number theory. Globally valued fields are an algebraic structure introduced by Ben Yaacov and Hrushovski encoding the data necessary for a theory of heights. They are closely related to Gubler's notion of M-fields, see \cite{gubler_m-fields}, and Chen and Moriwaki's adelic curves, see \cite{Adelic_curves_1}. Their algebraic nature makes GVFs particularly suitable for classification purposes and the application of methods from continuous logic. Many of the results in this section can be found in \cite{basics_of_gvfs}.

\subsection{Definitions of globally valued fields}

GVFs admit several equivalent definitions. Depending on the context, it is convenient to adopt different perspectives. For instance, the approach using GVF functionals or heights is advantageous for taking ultraproducts while viewing globally valued fields as equivalence classes of adelic curves allows for local computations. We recall several equivalent definitions of GVFs. For proofs of the equivalence of these definitions, we refer to \cite{basics_of_gvfs}.

Arguably, the simplest definition of a GVF consists of the axiomatization of the basic properties of the logarithmic Weil height. In particular, it follows immediately that $\Qbar$ is endowed with the structure of a globally valued field. 

\begin{definition}
    A \emph{globally valued field} (abbreviated GVF) is a field $K$ together with a height function $h:\coprod_{n\geq 1} K^n \rightarrow \mathbb{R}\cup\{-\infty\}$, satisfying the following axioms, for some fixed \emph{Archimedean error} $e \geq 0$.
\[
    \begin{array}{lll}
        \textnormal{Height of zero:} & \forall x\in K^n, & h(x) = -\infty \Leftrightarrow x=0 \\
        \textnormal{Height of one:} & & h(1,1) = 0 \\\textnormal{Product formula:} & \forall x\in K^{\times}, & h(x) = 0 \\
        \textnormal{Invariance:} & \forall x\in K^n,\, \forall \sigma\in \Sym_n, & h(\sigma x) = h(x) \\
        \textnormal{Additivity:} & \forall x\in K^n,\, \forall y\in K^m, & h(x\otimes y) = h(x) + h(y) \\
        \textnormal{Monotonicity:} & \forall x\in K^n,\, \forall y\in K^m, & h(x) \leq h(x,y) \\
        \textnormal{Triangle inequality:} & \forall x,y\in K^n, & h(x+y) \leq h(x,y) + e
    \end{array}
\]
    Here $\otimes$ denotes the Segre product, defined by $(x_1, \dots, x_n) \otimes (y_1, \dots, y_m) = (x_i \cdot y_j : 1 \leq i \leq n, 1 \leq j \leq m)$. Such a height function factors through $h:\mathbb{P}^n(K)\rightarrow \mathbb{R}_{\geq 0}$ for each $n$. We set $\height(x) := h[x:1]$ for $x \in K$.
\end{definition}

The additivity property suggests passing to the Grothendieck group of $\coprod_{n \geq 0} K^n\setminus \{0\}$ with the Segre product and realizing the GVF structure as a linear functional subject to suitable axioms. To be precise, one can associate to a field $K$ the space of lattice $\bbQ$-divisors $\LdivQ(K)$ with a well-defined notion of principal divisors and effective cone, see \cite[Definition 5.1]{basics_of_gvfs}. It furthermore has the structure of a lattice group, i.e.\ it has a join operation $\vee$ that corresponds to taking the maximum of two elements and is compatible with the group structure. 

Then, a GVF functional is a linear map $\phi:\Ldiv(K) \to \bbR$ nonnegative on the effective cone and $0$ on principal divisors. By positivity, the functional extends uniquely to the completion with respect to the effective cone which we denote by $\Mdiv(K)$ to express its maximality, cf.\ \cite[Proposition 3.1.16]{szachniewicz2023}. It inherits the structure of a lattice group.

If $U$ is a quasi-projective variety over $\bbZ$ with function field $L \subset K$ then one may view the space of adelic divisors $\aDivR(U)$, defined in \cite{yuan_zhang_adeliclinebundlesquasiprojective} see also \S \ref{sec:intersection_theory}, as a subset of $\Mdiv(K)$. Conversely, for every $D \in \LdivQ(K)$ there is a quasi-projective variety $U$ with function field $L \subset K$ such that $D$ is defined by an element of $\aDivR(U)$. Let $L$ be a subfield of $F$. Then, one may define a GVF structures on $K$ are in one-to-one correspondence with positive functionals
\[ \aDivR(K/L) \coloneqq \varinjlim \aDivR(\sX) \to \bbR, \]
where the colimit goes over all projective morphisms of quasiprojective varieties $\sX \to \sS$ together with compatible embeddings $\Frac(\sX) \subseteq K$ and $\Frac(\sS)\subseteq K$.

The space $\aDivR(K/L)$ is closed under the lattice operation on $\Mdiv(K)$ by applying the arguments in \cite[\S 3.1]{szachniewicz2023} to the present situation.

\begin{lemma}\label{lemm:big_line_bundle_triviality_of_GVF}
    Let $K$ be a finite type GVF extension of the trivial GVF $k$ or of $\bbQ$ with arbitrary GVF structure. Denote the associated GVF functional by $\phi$. Let $\sL$ be an (arithmetically) big line bundle on a quasi-projective variety $U$ with function field $K$ and suppose that $\phi(\sL) = 0$. Then, the GVF structure on $K$ is trivial.
\end{lemma}

\begin{proof}
    Let $\sL'$ be an arbitrary adelic line bundle on a variety $U'$ with function field $K$. Then, the inequality $\sL' \leq N\sL$ holds for sufficiently large $N$, i.e.\ the line bundle $\sL'^{-1} \otimes \sL^{\otimes N}$ defined over a variety $U''$ dominating $U$ and $U'$ admits a (small) section. Hence, $\phi(\sL') \leq N\phi(\sL) = 0$. The analogous inequality for $-\sL'$ implies that $\phi(\sL') = 0$.
\end{proof}

If $\Field$ is countable, GVF structures can also be described as equivalence classes of proper adelic curve structures on $\Field$ (originally defined in \cite{Adelic_curves_1}). 

\begin{definition}
    A \textit{proper adelic curve} is a field $\Field$ together with a measure space $(\Omega, \mathcal{A}, \nu)$ and with a map $(\omega \mapsto |\cdot|_{\omega}):\Omega \to M_{\Field}$ to the space of absolute values on $\Field$, such that for all $a \in \Field^{\times}$ the function 
    \[ \omega \mapsto \omega(a) := -\log|a|_{\omega} \]
    is in $L^1(\nu)$ with integral zero.
\end{definition}

If $\Field$ is equipped with a proper adelic curve structure, then one can define a GVF structure on it, by setting
\[ h(x_1, \dots, x_n) := \int_{\Omega} -\min_i(\omega(x_i)) d\nu(\omega). \]
Conversely, if $\Field$ is countable, any GVF structure on $\Field$ is represented by some proper adelic curve structure in this way (\cite[Corollary 1.3]{basics_of_gvfs}). If $\Field$ is uncountable, the GVF may not be representable by an adelic curve, but it is still represented by a topological adelic curve (defined in \cite{sedillot_topological_adelic_curves}).

Any field admits a trivial GVF structure by setting the height of every non-zero vector to be $0$. The logarithmic Weil height provides a GVF structure on any number field and therefore on $\Qbar$. By \cite[Lemma 10.2]{basics_of_gvfs}, this is essentially unique in the sense that for any $\lambda \geq 0$ there exists a unique GVF structure $\Qbar[\lambda]$ with $\height(2) = \lambda\log 2$. Similarly, for any field $k$ there is a unique GVF structure on $\ktbar$ denoted by $\ktbar[\lambda]$ whose restriction to $k$ is trivial and satisfies $\height(t) = \lambda$. Given a GVF $K$, we denote by $K[\lambda]$ the GVF structure obtained by scaling all heights by $\lambda$.

Given a GVF $K$ and an algebraic extension $E$ of $K$, there is a canonical way to endow $E$ with a GVF structure called the \emph{symmetric} extension which we denote by $E^{\sym}$, see \cite[Proposition 10.5]{basics_of_gvfs}. For Galois extensions, we obtain the unique Galois invariant height function.

A GVF is said to be \emph{non-archimedean} or \emph{geometric} if it satisfies the strong triangle inequality $h(x+y) \leq h(x,y)$. Otherwise it is called \emph{archimedean}. For any non-archimedean GVF $K$ the set $K_{\textrm{const}} = \{x\in K\ |\ \height(x) = 0 \}$ forms a field called the \emph{constant field} of $K$.

A GVF is archimedean if and only if, for a representing (topological) adelic curve, there is a positive measure of archimedean (pseudo-)absolute values.

\begin{lemma}\label{lemm:constant_field_symmetric_extension}
    Let $K$ be a non-archimedean GVF with constant field $k$. Then, its algebraic closure $\ov{K}$ endowed with the symmetric extension has constant field $\ov{k}$.
\end{lemma}

\begin{proof}
    It suffices to prove the statement for finite subextensions $L \subset \ov{K}$ and we may further restrict to the case that $L$ is a Galois extension or an inseparable extensions of degree $p$.

    If $L$ is a finite Galois extension, and $x \in L$ belongs to the constant field, then also all Galois conjugates $\sigma x$ belong to the constant field. The symmetric polynomials in the conjugates also belong to the constant field and thus to $k$. It follows that $x$ is algebraic over $k$. If $L$ is an inseparable extension with $L^p = K$ and $x \in L$ is in the constant field then $x^p$ is in the constant field and $x$ is algebraic over $k$.
    
    Conversely, any element in $L$ which is algebraic over $k$ has to belong to the constant field. Let $\alpha \in L$ be finite over $k$. Then, the finite extension is of the form $k[\alpha] = \bigoplus^{d-1}_{i=0} k \alpha^i$ as a $k$-vector space. By the strong triangle inequality it follows that the height of elements in $k[\alpha]$ is bounded. However, if $\alpha$ has height $> 0$ the height of $\alpha^n$ is unbounded.
\end{proof}

\subsection{\label{sec:intersection_theory}Intersection numbers over GVFs}

Basic definitions and results concerning intersections of line bundles over globally valued fields were obtained in \cite{continuity_of_heights}. The definitions are compatible both with intersection numbers on adelic curves, see \cite{Adelic_curves_2}, as well as the Deligne pairing defined in \cite{yuan_zhang_adeliclinebundlesquasiprojective}. We will base our discussion on the theory of Yuan and Zhang which we only recall briefly. We refer to \cite{yuan_zhang_adeliclinebundlesquasiprojective} for details.

An arithmetic $\bbQ$-divisor $\sD$ on an integral projective scheme $X$ over $\Spec \bbZ$ consists of a $\bbQ$-Cartier divisor $D$ on $X$ and if $X$ is flat over $\Spec \bbZ$ a continuous Green's function $g_{\sD}$ for $D$ on $X(\bbC)$. Denote the space of arithmetic $\bbQ$-divisors on $X$ by $\aDivQ(X)$. Let $U$ be a quasi-projective scheme over $\Spec \bbZ$. Then, the space of model divisors on $U$ is defined as the cofiltered limit
\[
    \aDivQ(U)_{\textrm{mod}} = \varinjlim \aDivQ(\ov{U}),
\]
where $\ov{U}$ ranges over all compactifications of $U$. Given a compactification $\ov{U}$ a boundary divisor $\sE$ is a model divisor whose support is $\ov{U}\setminus U$ and satisfies that $g_{\sE} > 0$ everywhere. This allows us to define a boundary norm
\[
    \|\blank\|_{\sE}:\aDivQ(U)_{\textrm{mod}} \to [0, \infty]
\]
by
\[
    \sD \mapsto \inf \{\lambda \in \bbQ_{>0}\ |\ \lambda \sE - \sD \textrm{ and } \lambda \sE + \sD \textrm{ are effective}\}.
\]

The space of adelic divisors on $U$ is defined to be the completion of the space of model divisors with respect to the boundary norm. It is denoted by $\aDivQ(U)$.

An adelic divisor is called \emph{strongly nef} if it can be approximated by nef model divisors. An adelic divisor is called \emph{nef} if it is in the closure of the strongly nef cone. An adelic divisor is called \emph{integrable} if it is the difference of two strongly nef divisors. We can specify the cone of strongly nef/nef/integrable divisors by a subscript snef/nef/int. Equivalence classes of adelic divisors induce a notion of adelic line bundles which we denote by $\aPic$ with suitable adornment. We are now in the position to define adelic line bundles on quasi-projective varieties over a GVF $K$. Let $K$ be a GVF and $X$ a quasi-projective variety over $K$. We define a model of $X$ to consist of a finite type normal quasi-projective scheme $\sY$ over $\Spec \bbZ$, a flat projective morphism $\sX \to \sY$, a morphism $\Spec K \to \sY$ mapping to the generic point of $\sY$ and an identification $X \cong \sX_K$. Any two models of $X$ are dominated by a third. An \emph{adelic line bundle} on $X$ consists of an adelic line bundle on a model $\sX$. We identify adelic line bundles if they agree after pullback to a model dominating both models in question. We call an adelic line bundle on $X$ \emph{strongly nef/integrable} if the line bundle on the model corresponding to it is. We call an adelic line bundle on $X$ \emph{semipositive} if for a model $(\sX \to \sY, \ov{\sL})$ and any $y \in \sY^{\an}_{\textrm{Berk}}$ the restriction of $\ov{\sL}$ to any fibre $\sX^{\an}_y$ is semipositive.

We recall the Deligne pairing of adelic line bundles in the theory of Yuan and Zhang in order to define arithmetic intersection numbers over GVFs. Let $f:\sX \to \sY$ be a flat projective morphism of quasi-projective schemes over $\Spec \bbZ$ and assume that $\sX$ is integral and $\sY$ is normal. The Deligne pairing is introduced in \cite[Theorem 4.1.3]{yuan_zhang_adeliclinebundlesquasiprojective} as a pairing
\[
    \aPic(\sX)_{\textrm {int}}^{n+1} \to \aPic(\sY)_{\textrm {int}}.
\]
For integrable adelic line bundles $\ov{\sL}_1,\dots,\ov{\sL}_{n+1}$ the Deligne pairing is denoted by \[\langle\ov{\sL}_1,\dots,\ov{\sL}_{n+1}\rangle_{\sX/\sY}.\]

We define intersection numbers of integrable adelic line bundles over a GVF by computing the Deligne pairing on a suitable model and evaluating the height of the resulting line bundle.

\begin{definition}\label{def:arithmetic_intersection}
    Let $\ov{\sL}_1,\dots,\ov{\sL}_{n+1}$ be integrable adelic line bundles on a projective variety $X$ over a GVF $K$. Let $\sX \to \sY$ be a model of $X$ over which all the line bundles are defined. Let $\phi$ denote the GVF functional of $K$ restricted to $\Frac(\sY)$. Then we define the intersection number as
    \[
        \ov{\sL}_1\cdot\ldots\cdot\ov{\sL}_{n+1} = \phi(\langle\ov{\sL}_1,\dots,\ov{\sL}_{n+1}\rangle_{\sX/\sY}).
    \]
    The result does not depend on the choice of model by the compatibility of the Deligne pairing with base change in \cite[Theorem 4.1.3]{yuan_zhang_adeliclinebundlesquasiprojective}.
\end{definition}

An adelic line bundle $\ov{\sL}$ on a quasi-projective variety $U$ over a field $K$ has a geometric part $\sL$ that is a limit of line bundles on compactifications of $U$ over $K$. We refer to $\sL$ as a \emph{geometric adelic line bundle}. They are treated on equal footing with the arithmetic case in \cite{yuan_zhang_adeliclinebundlesquasiprojective}.

If $X$ is a projective variety over a GVF $K$ and $\ov{\sL}$ is an adelic line bundle on $X$ whose underlying line bundle is ample. Then, the normalized height of $Z$ is defined as
\[
    h_{\ov{\sL}}(Z) = \frac{(\ov{\sL}|_Z)^{\dim Z + 1}}{(\dim Z + 1) (\sL|_Z)^{\dim Z }}.
\]
This definition is used in order to make the intersection number comparable to the height of points on the variety. In particular, it is comparable to the essential and absolute minimum. If $\ov{\sL}$ is a canonically metrized ample symmetric line bundle on an abelian variety we will use the notation $\NTht(Z)$ for $h_{\ov{\sL}}(Z)$.

An adelic line bundle on a projective variety $X$ over a GVF $K$ induces an adelic line bundle on an adelic curve representing $K$ in the sense of Chen-Moriwaki. We will refer to adelic line bundles as defined by Chen and Moriwaki as \emph{C-M line bundles} and reserve the word {adelic line bundles} to the line bundles discussed above. Concretely, the analytification functor of Yuan and Zhang, see \cite[\S 3]{yuan_zhang_adeliclinebundlesquasiprojective}, associates to an adelic line bundle $\ov{\sL}$ on a quasi-projective variety $\sX$ over $\Spec \bbZ$ a continuously metrized line bundle $\ov{\sL}^{\an}$ on the Berkovich analytification $\sX^{\an}$. If $\ov{\sL}$ is an adelic line bundle on a projective variety $X$ over a GVF $K$, one may obtain an adelic line bundle in the sense of Chen and Moriwaki by choosing a model $(\pi:\sX \to \sY, \ov{\sL})$ and determining the local metrics by the metrics on the fibres of $\pi^{\an}:\sX^{\an}\to \sY^{\an}$. We denote the induced C-M line bundle by $\ov{\sL}^{\textrm{C-M}}$.

The paper \cite{Adelic_curves_2} contains a definition of arithmetic intersection numbers for C-M line bundles. Let $\ov{\sL}_1,\dots,\ov{\sL}_{n+1}$ be integrable adelic line bundles on a projective variety $X$ of dimension $n$ over a GVF $K$. By Proposition \ref{prop:C-M_YZ_identity}, one may compute their intersection number by choosing a representing adelic curve of a countable GVF over which the data is defined and evaluating the intersection number $\ov{\sL}_1^{\textrm{C-M}}\cdot\ldots\cdot\ov{\sL}_{n+1}^{\textrm{C-M}}$. This allows us to freely shift our perspective between the two approaches.

\begin{remark}
    A drawback of the present definition of the intersection number is the small class of line bundles for which it is defined. It imposes a positivity condition on the line bundle on $\sX$ while one only needs a relative condition on its positivity, as found in \cite{continuity_of_heights}. Moreover, our definition of adelic line bundle does not depend on the GVF structure on $K$. It is natural to complete the space of line bundles further after fixing a GVF structure. For our intended applications and many others this does not play a role. Both mentioned issues are remedied by C-M line bundles.
\end{remark}

In order to check that all line bundles of interest on semiabelian varieties are integrable we study the interaction of the introduced positive cones with the lattice structure on the space of line bundles.

\begin{lemma}\label{lemm:maxima_of_divisors_and_positivity}
    Let $\ov{\sD}_1, \ov{\sD}_2$ be nef/semipositive/integrable adelic divisors on a quasi-projective arithmetic variety $U$. Then, the divisor $\ov{\sD}_1 \vee \ov{\sD}_2$ is also nef/\allowbreak semipositive/\allowbreak integrable.
\end{lemma}

\begin{proof}
    The integrable case follows from the nef case. If $\ov{\sD}_1, \ov{\sD}_2$ are integrable we may find a nef divisor $\ov{\sE}$ such that both $\ov{\sD}_1 + \ov{\sE}$ and $\ov{\sD}_2 + \ov{\sE}$ are nef. Then, $(\ov{\sD}_1 + \ov{\sE}) \vee (\ov{\sD}_2 + \ov{\sE}) = (\ov{\sD}_1 \vee \ov{\sD}_2) + \ov{\sE}$ is nef. For the semipositive case one may assume that $\ov{\sD}_1, \ov{\sD}_2$ are semiample.

    It follows from the characterization of semipositive metrics on semiample line bundles in \cite[\S 3.3]{Chen_Moriwaki_semipositive_extensions} that the $\ov{\sD}_i$ are semipositive over a place precisely if as metrized divisors they can be uniformly approximated by divisors of the form $\frac{1}{N}(\bigvee \adiv(x_i) + (0,C))$, where $(0,C)$ is the trivial divisor with constant Green's function for some $C \in \bbR$. The maximum of two such divisors is again semipositive. Nef model divisors are precisely those that can be approximated by divisors of the form $\frac{1}{N}\bigvee \adiv(x_i)$, this follows for instance from \cite[Theorem 3.24]{continuity_of_heights}. The assertion follows analogously to the semipositive case.
\end{proof}

For our applications, the line bundles of importance will be canonically metrized line bundles on families of semiabelian varieties. All occurring line bundles are obtained by pullback and maximization from canonically metrized ample symmetric and numerically trivial line bundles on abelian varieties. As such it suffices to prove that these are integrable to prove that all  occurring line bundles are integrable.

A canonically metrized ample symmetric line bundle on an abelian variety is nef and thus integrable by \cite[Theorem 6.1.1]{yuan_zhang_adeliclinebundlesquasiprojective}. Canonically metrized numerically trivial line bundles on families of abelian varieties are also integrable.

\begin{lemma}\label{lemm:numerically_trivial_are_differences}
    Let $\pi:\sA \to S$ be a family of abelian varieties and let $\ov{\sQ}$ be a numerically trivial canonically metrized line bundle on $\sA$. Then, the adelic line bundle $\ov{\sQ}$ is integrable.
\end{lemma}

\begin{proof}
    Let $\ov{\sM}$ be an ample symmetric line bundle on $\sA$. Since it defines a polarization of $\sA$ this gives rise to a section $\sigma: S \to \sA$ and an isomorphism $\tr_\sigma^* \sM - \sM \cong \sQ$ of line bundles on $\sA$. We claim that $\ov{\sQ} = \tr_\sigma^* \ov{\sM} - \ov{\sM} - \pi^*\sigma^* \ov{\sM}$ as adelic line bundles. This is a family version of \cite[Corollaire 2.2]{chambert-loir_semiabelian_small_points} and follows using the same arguments replacing the theorem of the cube over a field by a relative version. It follows that $\ov{\sQ}$ is integrable since $\ov{\sM}$ is nef.
\end{proof}

\subsection{\label{sec:polarized_GVF}Polarized GVF structures}

An important class of GVF structures is provided by polarized GVF structures. They are a common generalization of the examples 3.2.4 and 3.2.6 in \cite{Adelic_curves_1}.

\begin{definition}\label{def:polarized_gvf}
    Let $K$ be a GVF and let $F$ be a field extension of finite type of transcendence degree $d$. A polarization of $F$ consists of a normal projective variety $X$ whose function field is $F$ and nef adelic line bundles $\ov{\sH}_1, \dots, \ov{\sH}_d$ on $X$. A polarization $(X,\ov{\sH}_1, \dots, \ov{\sH}_d)$ induces a GVF structure on $F$ by the functional
    \[
    \underset{\begin{subarray}{c}
  \pi:X' \to X\\ \textrm{birational}
  \end{subarray}}{\varinjlim} \aDivR(X) \to \bbR,
    \]
    defined by $l(\ov{D}) = \pi^*\ov{\sH}_1\cdot \ldots\cdot \pi^*\ov{\sH}_d\cdot \ov{D}$. We denote it by $K(X,\ov{\sH}_1, \dots, \ov{\sH}_d)$.
\end{definition}


This GVF structure admits an explicit description in terms of adelic curves. Let $(K,(\Omega,\sA,\nu),\phi)$ be a countable adelic curve. Let $X$ be an integral projective variety of dimension $d$ over $K$.

\begin{definition}[Definition 3.1 \cite{Adelic_curves_4}]
    Let $\Omega' \subset \Omega$ be a measurable subset. The \emph{global adelic space} of $X$ over $\Omega'$ is defined as $X_{\Omega'}^{\an} = \coprod_{\omega \in \Omega'} X_\omega^{\an}$. Denote by $\pi:X_{\Omega'}^{\an} \to \Omega'$ the map sending elements of $X_\omega^{\an}$ to $\omega$.

    This space is endowed with the smallest $\sigma$-algebra $\sB_{X,\Omega'}$ such that
    \begin{enumerate}
        \item the map $\pi:X_{\Omega'}^{\an} \to \Omega'$ is measurable,
        \item for any Zariski open subset $U \subseteq X$, the set $U_{\Omega'}^{\an}$ belongs to $\sB_{X,\Omega'}$
        \item for any Zariski open subset $U \subseteq X$ and any measurable family\footnote{This is defined in \cite[Definition 2.2]{Adelic_curves_4}.} $f$ of continuous functions over $U$ the induced function $f_{\Omega'}: U_{\Omega'}^{\an} \to \bbR$ is $\sB_{X,\Omega'}|_{U_{\Omega'}^{\an}}$-measurable.
    \end{enumerate}
    In particular, Green's functions for adelic divisors are $\sB_{X,\Omega}$-measurable.
\end{definition}

Let $\ov{\sH}_1, \dots, \ov{\sH}_d$ be arithmetically nef C-M line bundles. These give rise to a Borel measure family $\omega \mapsto c_1(\ov{\sH}_1,\omega)\dots c_1(\ov{\sH}_d,\omega)$, see \cite[Definition 2.17 and Example 2.20]{Adelic_curves_4}. This defines a measure $\mu_\Omega$ on $X_\Omega^{\an}$ by \cite[Definition 3.6]{Adelic_curves_4}. Concretely, it is determined by the property that if $f:X_\Omega^{\an} \to [-\infty,\infty]$ is the collection of Green's function of a C-M divisor, then $f$ is integrable and the integral satisfies
\[
    \int_{X_\Omega^{\an}} f(x) \mu_\Omega(dx) = \int_\Omega \left( \int_{X_\omega^{\an}} f_\omega(x)c_1(\ov{\sH}_1,\omega)\dots c_1(\ov{\sH}_d,\omega)\right)\nu(d \omega).
\]
One observes that $\mu_\Omega(Z_\Omega^{\an}) = 0$ for any closed subvariety $Z \subset X$. Since $K$ was assumed countable we have $\mu_\Omega(\bigcup_{Z \subset X \textrm{ cld}}Z_\Omega^{\an}) = 0$. Let $S = X_\Omega^{\an}\setminus \bigcup_{Z \subset X \textrm{ cld}}Z_\Omega^{\an}$.

Let $\operatorname{PDiv}(X)$ denote the set of prime divisors
\[
    \{ V \subset X\ |\ \textrm{closed, irreducible, }\codim(V)=1\}.
\]
Let $\sB_{\operatorname{PDiv}}$ be the discrete $\sigma$-algebra on $\operatorname{PDiv}(X)$. We define a measure $\mu_0$ on $\operatorname{PDiv}(X)$ by $\mu_0(\{D\}) = \adeg(\ov{\sH}_1\cdot \ldots\cdot \ov{\sH}_d | D)$.  We define a measure space $(\Omega_F, \sA_F, \mu) = (S \sqcup \operatorname{PDiv}(X), \sB_{X,\Omega}|_S \sqcup \sB_{\operatorname{PDiv}}, \mu_\Omega|_S \sqcup \mu_0)$. Points in $S$ naturally give rise to absolute values on $F$. The order of vanishing at $D \in \operatorname{PDiv}(X)$ defines a valuation. This allows us to define a map $\psi:\Omega_F \to M_F$ to the absolute values of $F$. The data $(F, (\Omega_F, \sA_F, \mu), \psi)$ defines a proper adelic curve. It follows from \cite[Theorem 4.2.11]{Adelic_curves_2} that it induces the polarized GVF structure as in Definition \ref{def:polarized_gvf} provided that $\ov{\sH}_1, \ldots, \ov{\sH}_d$ are adelic line bundles.

\begin{remark}
    The restriction of the polarized GVF structure on $F$ to $K$ is precisely $K[\sH_1\cdot \ldots\cdot \sH_d]$, where $\sH_1\cdot \ldots\cdot \sH_d$ denotes the geometric intersection number of the underlying line bundles. If $X$ is $0$-dimensional, the polarized GVF structure is $F^{\sym}[[F:K]]$.
\end{remark}

\begin{example}
    Let $k$ be an algebraically closed field and $X$ a variety of dimension $d$ over $k$ and big nef line bundles $L_1, \dots, L_{d-1}$. Then, this defines a GVF structure on $k(X)$ with constant field $k$. A different choice of big nef line bundles is comparable in the sense that the height functions $h$ and $h'$ satisfy $\frac{1}{C}h \leq h' \leq C h$ for some $C > 1$. We will refer to such a GVF structure on $k(X)$ as an \emph{ample GVF structure}.
\end{example}

\begin{propdef}
    Let $(X,L,f)$ be a polarized abelian variety over an algebraically closed GVF $K$, where $X$ is normal. Then, we define the canonical GVF structure on $K(X)$ to be $K(X, \ov{L}, \dots,\ov{L})[((d+1)L^d)^{-1}]$ where $\ov{L}$ is endowed with the canonical metrics. It is a GVF extension of $K$. We also refer to the symmetric extension of this GVF structure to $K(X)$ as the canonical GVF structure. This is independent of the choice of polarization.
\end{propdef}

\begin{proof}
    There exists a Zariski dense set of preperiodic points which are of height $0$ for any choice of polarization. Let $(x_i)_{i\in I}\in X(K)$ be a generic net of preperiodic points and any polarization $L$ one has 
    \[\lim_i h_{\ov{M}}(x_i) = \frac{\ov{L}^d\cdot \ov{M}}{(d+1)\deg_L(X)}\]
    by Theorem \ref{thm:equidistribution}. In particular, the GVF functional must agree on every $\ov{M}$ that can be defined over $X$. Suppose that $l$ is a polarized GVF functional and $l'$ be any GVF functional that agrees on line bundles $\ov{M}$ on $X$. Then, we notice that $l' - l$ is a GVF that restricts to the trivial GVF structure on $K$. Moreover, given an adelic line bundle $\ov{M}$ on $X$ whose underlying line bundle is ample we have $(l' - l)(\ov{M}) = 0$ and hence $l' = l$ by Lemma \ref{lemm:big_line_bundle_triviality_of_GVF}.
\end{proof}

\begin{theorem}[Proposition 5.5 \cite{A_Sedillot_diff_of_relative_volume}]\label{thm:equidistribution}
    Let $X$ be a projective variety defined over an adelic curve $(K,(\Omega, \sA,\nu), (|\blank |_{\omega}))$. Let $\ov{L}$ be a semipositive C-M line bundle whose underlying line bundle is ample and satisfies $h_{\ov{L}}(X) = 0$. Let $(x_i)_{i \in I} \in X(K)$ be a net of points satisfying
    \[
    \lim_{i \in I} h_{\ov{L}}(x_i) = 0.
    \]
    Then, for any C-M line bundle $\ov{M}$ on $X$ one has
    \[
        \lim_i h_{\ov{M}}(x_i) = \frac{\ov{L}^d\cdot \ov{M}}{(d+1)\deg_L(X)}.
    \]
\end{theorem}

\subsection{Ultraproducts of GVFs}

We introduce ultraproducts of globally valued fields and discuss their basic properties based on \cite{basics_of_gvfs}. This construction fits naturally into the framework of (continuous) model theory, cf.\ \cite[Section 11]{basics_of_gvfs}. 

\begin{definition}
    A \emph{filter} on a set $I$ is a family $\sU \subseteq 2^I$ satisfying the following conditions.
    \begin{enumerate}
        \item Non-degeneracy: $\emptyset \notin \sU$.
        \item Upward closedness: If $A \in \sU$ and $A\subseteq B \subseteq I$, then $B \in \sU$.
        \item Closedness under intersections: If $A, B \in \sU$, then $A\cap B \in \sU$.
    \end{enumerate}
    If $\sU$ is maximal in the sense that for any $A \subseteq I$ either $A \in \sU$ or $I\setminus A \in \sU$, then we call $\sU$ an \emph{ultrafilter}. Equivalently, an ultrafilter may be defined as a finitely additive measure $\mu:2^I \to \{0,1\}$ with $\mu(I) =1$. We call a subset $J\subseteq I$ big if $J \in \sU$ or equivalently $\mu(J) = 1$.
\end{definition}

\begin{lemma}[Ultrafilter lemma]
    Assuming the axiom of choice, every filter on $I$ is contained in an ultrafilter on $I$.
\end{lemma}

We will freely use the ultrafilter lemma to construct ultrafilters. We remark that the ultrafilter lemma is strictly weaker than the axiom of choice.

Let $\beta I$ denote the set of ultrafilters on $I$. We endow it with the topology generated by the sets $U_E = \{\sU\ |\ E \in \sU\}$ for all $E\subseteq I$. The resulting topological space is compact and Hausdorff. It satisfies the universal property that for any compact Hausdorff space $K$ with underlying set $|K|$ one has
\[
    \Hom_{\mathsf{Top}}(\beta I,K) = \Hom_{\mathsf{Set}}(I,|K|).
\]
Under this identity a map $f:I \to |K|$ corresponds to the map $\beta f:\beta I \to K$ which satisfies $\beta f(\sU)$ is given by the unique point $x$ such that for every open neighborhood $U$ of $x$ the preimage $f^{-1}(U)$ belongs to $\sU$. Writing the function $f$ as $(x_i)_{i\in I}$, we denote the point $x$ by $\lim_{\sU} x_i$ and refer to it as the limit of $(x_i)_{i\in I}$ along $\sU$.

\begin{definition}
    Let $I$ be a set, let $\sU$ be an ultrafilter on $I$, and let $(M_i)_{i \in I}$ be globally valued fields. We define the ultraproduct $M = \underset{\sU}{\prod}M_i$ as $M_0/\sim$ where
    \[
        M_0 = \{(x_i)\in \prod_I M_i\ |\ \height(x_i) \textrm{ is bounded}\}
    \]
    and two elements $(x_i)$ and $(x_i')$ of $M_0$ are called equivalent if $\mu(\{i\in I\ |\ x_i = x_i'\} ) = 1$. $M$ is endowed with the structure of a GVF. The algebraic operations are defined factor-wise on $M_0$ and descend to operations on $M$. The heights are defined as the limit of the factor-wise heights along the ultrafilter. 
\end{definition}

Let $X$ be a finite type scheme over $K$ and cover $X$ by finitely many affine opens which we each view as a closed subvariety of $\bbA^n$. We say that a collection of GVF valued points has bounded height if the height of their coordinates is bounded on each of the affine opens. A collection $x_i \in X(M_i)$ of GVF valued points of bounded height defines a point $x \in X(M)$ valued in the ultraproduct $M= \prod_\sU M_i$ of the $M_i$. If $X$ is embedded in a projective space the Weil height of $x$ agrees with the limit of the Weil height $\lim_\sU h(x_i)$ of the $x_i$. If $X$ is quasiprojective and $\ov{\sL}$ is a geometrically ample adelic line bundle on $X$ a family of GVF valued points $x_i \in X(M_i)$ is of bounded height if and only if $h_{\ov{\sL}}(x_i)$ is bounded.

Let $\sX$ be a finite type quasi-projective scheme over $\Spec \bbZ$ and let $\ov{\sL}$ be an adelic line bundle in the sense of \cite[\S 2.5]{yuan_zhang_adeliclinebundlesquasiprojective}. We define the height of a GVF-valued point as the limit of the heights of approximating model line bundles which are well defined since the Weil height of a projective tuple is defined over any GVF. Let $x_i \in \sX(F_i)$ be GVF-valued points of bounded height. Let $F = \prod_\sU F_i$ be the ultraproduct of the $F_i$ and let $x \in \sX(F)$ be the limit point. 

\begin{lemma}
    The height function is continuous in the sense that 
    \[
        \lim_\sU h_{\ov{\sL}}(x_i) = h_{\ov{\sL}}(x).
    \]
\end{lemma}

\begin{proof}
    When $\ov{\sL}$ is a model line bundle on a projective compactification the result follows from the continuity of the Weil height for points in projective space. By definition, for a boundary divisor $\ov{\sE}$ on a compactification and $\epsilon > 0$ there exists a model line bundle $\ov{\sL}_\epsilon$ such that $|h_{\ov{\sL}} - h_{\ov{\sL}_\epsilon}| < \epsilon h_{\ov{\sE}}$. Since the $x_i$ have bounded height, $h_{\ov{\sE}}$ is bounded on the collection of the $x_i$. Hence, the result follows by passing to the limit.
\end{proof}

\subsection{\label{sec:existential_closedness}GVF analytification and existential closedness}

The GVF analytification of a finite type scheme over a GVF $K$ was introduced in \cite[Definition 2.6]{continuity_of_heights} as a global analogue of the Berkovich analytification. As in the Berkovich setting, it can be interpreted as a space of quantifier-free types in continuous logic. This perspective underlies the equivalent definition in \cite[Construction 11.15]{basics_of_gvfs}. At the end of the section we give a down to earth treatment of existential closedness of GVFs. For a treatment in the context of continuous logic we refer to \cite[\S 12]{basics_of_gvfs}.

\begin{definition}
    Let $K$ be a GVF and let $X$ be a finite type scheme over $K$. We define the \emph{GVF analytification} of $X$ over $K$, denoted $\GVFan[K]{X}$ (or $\GVFan{X}$ if the base GVF is implied), to be the set of pairs $x = (\pi(x),h_x)$ where $\pi(x)$ is a point of $X$, and $h_x$ is a height function on $\kappa(\pi(x))$ extending the given height on $K$. If $x\in \GVFan[K]{X}$ and $U\subset X$ is an open containing $\pi(x)$, we also denote by $h_x$ the induced map
\[
    h_x : \begin{array}{ccc}
            \mathcal{O}_X(U)^n & \rightarrow & \mathbb{R}\cup\{-\infty\} \\
            (f_1,\ldots,f_n) & \mapsto & h_x(f_1(\pi(x)),\ldots,f_n(\pi(x)))
          \end{array}.
\]

We equip $\GVFan[K]{X}$ with the weakest topology for which
\begin{enumerate}
    \item The map $\pi : \GVFan[K]{X} \rightarrow X$ is continuous, where $X$ is endowed with the constructible topology.
    \item For every Zariski open $U\subset X$, and every tuple $(f_1,\ldots,f_n)\in \mathcal{O}_X(U)^n$, the map $x \mapsto h_x(f_1,\ldots,f_n)$ is continuous.
\end{enumerate}
\end{definition}

The resulting topological space is locally compact. If $X$ is quasi-projective one obtains a compact subset by imposing a bound on the height with respect to an ample line bundle on a compactification. Given a GVF extension $F/K$, there is an analytification map $X(F) \to \GVFan{X}$ by restricting the GVF structure to the residue field of the scheme theoretic image.

Given a set of points $(x_i \in X(F_i))_{i \in I}$ valued in GVF extensions $F_i$ of $K$ of bounded height each ultrafilter $\sU$ on $I$ induces a point $x_\sU \in X(\prod_\sU F_i)$. The induced map $\beta I \to \GVFan{X}$ is the unique continuous extension of the map $I \to \GVFan{X}$. An adelic line bundle $\ov{\sL}$ on $X$ induces a continuous function $h_\sL: \GVFan{X} \to \bbR$.

\begin{definition}
    A GVF $K$ is called existentially closed if for every finite type scheme $X$ over $K$ the analytification map $X(K) \to \GVFan{X}$ has dense image.
\end{definition}

\begin{theorem}[\cite{szachniewicz2023},\cite{benyaacov_hrushovski_existential_closure}]
    The globally valued fields $\Qbar[1]$ and $\ktbar[1]$ are existentially closed. 
\end{theorem}

These results allow us to reduce statements for general GVFs by appealing to the corresponding results for $\Qbar[1]$ or $\ktbar[1]$. To illustrate this strategy, we prove the following lemma which is a special case of \cite[Lemma 12.2]{basics_of_gvfs}. 

\begin{lemma}
    Any GVF extension $F/K$ of an existentially closed field $K$ can be embedded into an ultrapower $\prod_\sU K$.
\end{lemma}

\begin{proof}
    We first assume that $F$ is of finite type over $K$. Let $X$ be a model of $F$ over $K$. Then, the GVF structure on $F$ determines a point $\eta \in \GVFan{X}$. By the density of $K$-points in the analytification we may find a net of $K$-points converging to $\eta$. We extend this net to an ultrafilter $\sU$. Then, the $K$-points define a limit point in $X(\prod_\sU K)$ that maps to $\eta$ under analytification.

    If $F$ is not of finite type it can be embedded into an ultraproduct of its finite type subfields. Hence $F$ can be embedded into an ultrapower of $K$.
\end{proof}

Since every non-trivial GVF contains a copy of $\ktbar[\lambda]$ or $\Qbar[\lambda]$ for some $\lambda > 0$ we can embed every GVF into an ultrapower of these basic GVFs. Furthermore, $\ov{\bbQ(t)}[\lambda]$ can be embedded into an ultraproduct of $\Qbar[\lambda_i]$ with varying scaling. It follows that any GVF of characteristic $0$ can be embedded into an ultraproduct of $\Qbar[\lambda]$. Similarly, any GVF of characteristic $p$ can be embedded into an ultraproduct of $\ktbar[1]$ for $k$ of characteristic $p$.

We are now in the position to prove that the two alternative arithmetic intersection numbers defined following Yuan-Zhang and Chen-Moriwaki respectively agree. We extend this from the case of $\Qbar$ and $\ktbar$. It is more natural to prove this directly, but given the existing results in the literature this is the most efficient.

\begin{proposition}\label{prop:C-M_YZ_identity}
    Let $\ov{\sL}_1,\dots,\ov{\sL}_{n+1}$ be integrable adelic line bundles on a projective variety $X$ of dimension $n$ over a GVF $K$. Let $K' \subset K$ be a countable subfield over which the data is defined. Then, the arithmetic intersection number $\ov{\sL}_1\cdot\ldots\cdot\ov{\sL}_{n+1}$ as defined in Definition \ref{def:arithmetic_intersection} agrees with the intersection number $\ov{\sL}_1^{\textrm{C-M}}\cdot\ldots\cdot\ov{\sL}_{n+1}^{\textrm{C-M}}$ defined over an adelic curve representing $K'$.
\end{proposition}

\begin{proof}
    Both of the intersection numbers can be computed over $K'$ which was assumed to be of finite type over its prime field since intersection numbers are unaffected by extensions of the base field. After potential rescaling, we may assume the GVF $K'$ to contain either $\ktbar$ or $\Qbar$ as a sub-GVF.

    Let $\sX \to \sS$ be a model over which $\ov{\sL}_1,\dots,\ov{\sL}_{n+1}$ are defined and satisfying $\Frac(\sS) = K'$. Then, we consider two functions on $\GVFan{\sS} \to \bbR$. The first is $h_{\textrm{YZ}} = h_{\langle\ov{\sL}_1,\dots,\ov{\sL}_{n+1}\rangle_{\sX/\sS}}$ which defines a continuous function because it is the height with respect to an adelic line bundle. On the other hand, for any point $s \in \GVFan{\sS}$ one may consider the C-M line bundles $\ov{\sL}_i(s)^{\textrm{C-M}}$ on the fibre $\sX_s$ and define a map
    \[
        h_{\textrm{C-M}}: s \mapsto \adeg(\ov{\sL}_1(s)^{\textrm{C-M}}\cdot\ldots\cdot\ov{\sL}_{n+1}(s)^{\textrm{C-M}}\ |\ \sX_s).
    \]
    By \cite[Theorem 1.4]{continuity_of_heights}, this is independent of the choice of representing adelic curve and also a continuous function.
    
    It follows from \cite[\S 4.5]{Adelic_curves_2} that YZ intersection numbers and the C-M intersection numbers agree over $\Qbar$ and $\ktbar$. Since $\Qbar$- or $\ktbar$-points are dense in the GVF analytification it follows that the two functions agree. This concludes the result.
\end{proof}

An important question which often has deep model theoretic consequences is to axiomatize the existentially closed models of a logic theory. Such an axiomatization is known as the model companion of an (inductive) theory. The existence of a model companion for valued fields as a continuous logic theory was shown in \cite{benyaacov_continuous_model_theory_valued_fields}. The theory of globally valued fields is also conjectured to have a model companion.

\begin{conjecture}[Conjecture 12.7 \cite{basics_of_gvfs}]
    The theory of GVFs has a model companion. Concretely, the ultraproduct of existentially closed GVFs is existentially closed.
\end{conjecture}

Let us give a concrete application of the existence of a model companion to simplify the proof of Lemma \ref{lemm:semiab_intersection_small_implies_small_points}. One can embed $K$ into an ultrapower $F$ of an existentially closed field. Under the existence of the model companion, the GVF $F$ itself is existentially closed. The polarization $\pi^*\ov{\sM}^{\dim \pi(X)}\cdot \arho(\Delta)^{\dim X - \dim \pi(X)}$ induces a GVF structure on the function field $F(X)$. By the existential closedness of $F$, it follows that this point can be approximated by $F$-points, yielding the claim.

\subsection{\label{sec:zhang_inequalities}Families of subvarieties and Zhang inequalities}

We start out by recalling some foundational results on relative Hilbert schemes from \cite{grothendieck_hilbert-schemes}. This will be useful in order to study the essential minimum and its behavior in families.

\begin{theorem}
    Let $\sX \to S$ be a flat family of projective varieties over a noetherian scheme $S$. Let $\sL$ be a relatively very ample line bundle on $\sX$. Then there is a finite set of polynomials $\Xi$ such that for any reduced closed subvariety $Y \subset \sX_{\ov{s}}$ of some geometric fibre of $\sX$ of pure dimension $d$ and degree $r$ the Hilbert polynomial of $Y$ belongs to $\Xi$.
\end{theorem}

\begin{proof}
    Let $E_{d,r}$ be the the collection of coherent sheaves on geometric fibres $\sX_{\ov{s}}$ given by the structure sheaves $\sO_Y$ for reduced closed subvarieties $Y \subset \sX_{\ov{s}}$ of pure dimension $d$ and degree $r$. By Lemma 2.4 in \cite{grothendieck_hilbert-schemes} they form a limited family. By Theorem 2.1 b) loc.cit.\ the Hilbert polynomials associated to a limited family is finite.
\end{proof}

Given a polynomial $P$, the relative Hilbert scheme $\Hilbert_P(\sX/S)$ is a projective $S$-scheme  representing the functor
\begin{align*}
    \{\textrm{schemes over } S\} &\longrightarrow \Set \\
    T &\longmapsto
    \left\{
        \begin{array}{c}
        \textrm{closed subschemes } \\
        W \subset \sX \times_S T, \textrm{ flat over } T, \\
        \textrm{with Hilbert polynomial } P
        \end{array}
    \right\}.
\end{align*}

By definition there is a universal family $\mathfrak S \subset \sX\times_S \Hilbert_P(\sX/S)  \to \Hilbert_P(\sX/S)$. The restricted Hilbert scheme $\resHilbert_P(\sX/S)$ is defined as the locus where the fibres of $\mathfrak S \to \Hilbert_P(\sX/S)$ are geometrically integral. By \cite[IV Thm 12.2.4(vii)]{EGA}, the locus where the fibres of a flat, proper map are geometrically integral is open. In particular, $\resHilbert_P(\sX/S) \subset \Hilbert_P(\sX/S)$ is an open subscheme.

\begin{definition}
    Let $\ov{\sL}$ be an adelic line bundle on a projective geometrically integral variety $X$ over a GVF $K$. We define the absolute minimum of $X$ as
    \[
        \absmin(\ov{\sL}) = \underset{\begin{subarray}{c}
  F/K \\ \textrm{GVF extension}
  \end{subarray}}{\inf}\inf_{x \in X(F)} h_{\ov{\sL}}(x).
    \]
    We define the essential minimum to be
    \[
        \essmin(\ov{\sL}) = \underset{\begin{subarray}{c}
  F/K \\ \textrm{GVF extension}
  \end{subarray}}{\inf}\underset{\begin{subarray}{c}
  Z \subset X_F\\ \textrm{proper closed subset}
  \end{subarray}}{\sup}\inf_{x \in X(F) \setminus Z(F)} h_{\ov{\sL}}(x).
    \]
\end{definition}

\begin{remark}\label{rk:successive_minima_attained}
    If $K$ is existentially closed, the absolute and essential minimum can be computed without the need of passing to a GVF extension.

    There always exists a GVF extension $F/K$ that attains the infimum. Let $F_i$ be field extensions of $K$ such that for $i \to \infty$ the quantities $\inf_{x \in X(F)} h_{\ov{\sL}}(x)$ and
\[
    \underset{\begin{subarray}{c}
  Z \subset X_{F_i}\\ \textrm{proper closed subset}
  \end{subarray}}{\sup}\inf_{x \in X(F_i) \setminus Z(F_i)} h_{\ov{\sL}}(x).
    \]
    tend toward the absolute resp. the essential minimum. Then, for a non-principal ultrafilter $\sU$ the ultraproduct $\prod_{\sU} F_i$ attains the infimum. In fact, by taking ultraproducts one can find a GVF extension $F$ such that $\{x \in X(F)\ | h_{\ov{\sL}}(x) = \essmin(\ov{\sL}) \}$ is Zariski dense.
\end{remark}

\begin{example}
     In the definition of essential minimum, it is important to allow for subvarieties $Z$ defined over $F$ and not just over $K$. Consider the variety $ Z = \ov{\{(x,x+a)\}} \subset \bbP^1 \times \bbP^1$ defined over $\bbQ(a)$. Endow $\bbQ(a) \subset \bbQ(x,x+a)$ with a GVF structure such that $\height(x) = \height(x+a) = 0$ extending the natural GVF structure on $\bbQ(x,x+a)$. This can for instance be obtained by viewing $\bbQ(x,x+a)$ as the function field of $\bbP^1 \times \bbP^1$ with the polarized GVF structure by $\sO(1,1)^{\textrm{can}}$. Let $\ov{\sL} = \sO(1,1)^{\textrm{can}}|_Z$. Then, by the Bogomolov conjecture for tori in Proposition \ref{prop:archimedean_torus} it follows that $\essmin(\ov{\sL}) > 0$. However, the GVF structure on $\bbQ(x,x+a)$ describes a GVF valued point of height $0$ with respect to $\ov{\sL}$ that is transcendental over $\bbQ(a)$. This implies that $\ov{\bbQ(a)}^{\sym}$ is not existentially closed.
\end{example}

We introduce some helpful invariants that behave better in families. They will be defined in such a way that they are formulas in continuous logic, where the supremum and infimum over sets of bounded heights is allowed as a quantifier. By \L o\'s's theorem in continuous logic, see \cite{basics_of_gvfs}, the formula for the ultraproduct evaluates to the limit of the evaluation of the formula for each factor. Rather than appealing to continuous logic we prove that the invariants commute with taking ultraproducts in an ad hoc manner.

Let $\sX \to S$ be a family of projective varieties. We denote the fibre at a field-valued point $s \in S(F)$ by $\sX(s)$. Subvarieties of $\sX(s)$ of degree at most $N$ are parametrized by a finite union $\bbH$ of relative Hilbert schemes. Fix a compactification of $\bbH$ and an arithmetically ample line bundle $\sM$ on it. Let $\ov{\sL}$ be an adelic line bundle on $\sX$ that is geometrically relatively ample.

\begin{definition}
    Let $d,N,C > 0$ be constants. For any GVF-valued point $s \in S(F)$ define the set of closed subvarieties
    \[
        I(s,\ov{\sL},d,N,C) = \{Z \subset \sX\ |\ \dim(Z) < d-1, \deg_{\sL}(Z) \leq N, h_{\sM}([Z])\leq C\}.
    \]
    We define the parametrized minimum
    \[
        \zeta(s,\ov{\sL},d,N,C) = \underset{Z \in I(s,\ov{\sL},d,N,C)}{\sup}\inf_{x \in \sX(s)(F) \setminus Z(F)} h_{\ov{\sL}}(x).
    \]
\end{definition}

\begin{lemma}\label{lemm:parmin_ultraproducts}
    Let $\sX \to S$ be a family of varieties and let $\ov{\sL}$ be an adelic line bundle on $\sX$ that is geometrically relatively ample. Let $s_i \in S(F_i)$ be GVF-valued points of bounded height. Let $F = \prod_\sU F_i$ be an ultraproduct and $s\in S(F)$ the limit point of the $s_i$.
    Then, the parametrized minimum of $\ov{\sL}$ restricted to the fibres satisfies
    \[
        \lim_\sU \zeta(\ov{\sL}|_{\sX(s_i)},d,N,C) = \zeta(\ov{\sL}|_{\sX(s)},d,N,C).
    \]
    Consequently, the essential and absolute minimum satisfy \[\lim_\sU \absmin(\ov{\sL}|_{\sX(s_i)}) = \absmin(\ov{\sL}|_{\sX(s)})\] and \[\lim_\sU \essmin(\ov{\sL}|_{\sX(s_i)}) \geq \essmin(\ov{\sL}|_{\sX(s)}).\]
\end{lemma}

\begin{proof}
    Let $r < \zeta(\ov{\sL}|_{\sX(s)},d,N,C)$ be a real number. Then, there exists a subvariety $Z \subset \sX(s)$ of dimension $< d-1$, degree $\leq N$ and height $< C$ containing $\{x\in \sX(s)(F)\ |\ h_{\ov{\sL}}(x) < r\}$. By taking representatives, the point $P \in \bbH(F)$ given by $Z$ in the relative Hilbert scheme defines points $P_i \in\bbH(F_i)$ defining subvarieties $Z_i \subset \sX(s_i)$ of height $<C$ for a big subset of $i$. We remark that 
    \[
        r \leq \inf_{x \in X(s)(F) \setminus Z(F)} h_{\ov{\sL}}(x) = \lim_\sU \inf_{x \in X(s_i)(F_i) \setminus Z(F_i)} h_{\ov{\sL}}(x)
    \]
    since a point in $X(s)(F) \setminus Z(F)$ precisely corresponds to an equivalence class of points of bounded height $x_i \in X(s_i)(F_i) \setminus Z(F_i)$. In particular it follows that $\lim_\sU \zeta(\ov{\sL}|_{\sX(s_i)},d,N,C) \geq \zeta(\ov{\sL}|_{\sX(s)},d,N,C)$.

    Let conversely $r < \lim_\sU \zeta(\ov{\sL}|_{\sX(s_i)},d,N,C)$ be a real number. Then, there exist $r_i \in \bbR$ with $\lim_\sU r_i < r$ such that there exist subvarieties $Z_i \subset \sX(s)$ of dimension $< d-1$, degree $\leq N$ and height $< C$ containing 
    \[\{x_i \in \sX(s_i)(F_i)\ |\ h_{\ov{\sL}}(x_i) < r_i\}.\]
    Each $Z_i$ defines a point $P_i \in\bbH(F_i)$. Since they have bounded height they define a limit point $P \in \bbH(F)$ defining a subvariety $Z \subset \sX(s)$ and
    \[
        \inf_{x \in X(s)(F) \setminus Z(F)} h_{\ov{\sL}}(x) = \lim_\sU \inf_{x \in X(s_i)(F_i) \setminus Z(F_i)} h_{\ov{\sL}}(x) > r.
    \]
    This provides the reverse inequality.
\end{proof}

\begin{remark}
    It follows from Proposition \ref{prop:isotrivial_bogomolov}, which we will prove later, that the height on $\bbH$ associating to $[Z] \in \bbH(F)$ the height $h_{\ov{\sL}}(Z)$ is associated to a line bundle which is relatively ample over $S$. In particular, by defining the set of removable subvarieties
    \[
        I(\ov{\sL},d,N,C) = \{Z \subset X\ |\ \dim(Z) < d-1, \deg_{\sL}(Z) \leq N, h_{\ov{\sL}}(Z)\leq C\}.
    \]
    one may obtain a quantity that behaves the same way in families, but only relies on the fibre for its definition.
\end{remark}

\begin{proposition}\label{prop:lower_zhang_ineq}
    Let $\ov{\sL}$ be a geometrically ample semipositive adelic line bundle on a projective variety $X$ over a GVF $K$. Then,
    \[
        \frac{d}{d+1}\absmin(\ov{\sL})+ \frac{1}{d+1}\essmin(\ov{\sL}) \leq h_{\ov{\sL}}(X).
    \]
\end{proposition}

\begin{proof}
    Spread out the data to a family $\sX \to S$ over $\Qbar$ or $\ktbar$ such that wlog.\ $S$ has function field $K$. Let $s \in S(K)$ be such that $X = \sX(s)$. By existential closedness one may approximate $s$ by $\Qbar$- or $\ktbar$-points. By the properties of the height, the absolute and the essential minimum under taking ultraproducts the inequality follows from the case of $\Qbar$ or $\ktbar$ which is known from \cite[Theorem 5.2]{Zhang_thesis_inequality}.
\end{proof}

\begin{remark}
    This does not contradict the counterexample in \cite[Theorem 1.2]{guo2025nefconesuccessiveminima}. This is because we allow passing to GVF extensions in our definition of successive minima. From our perspective, it is expected that the naive version of Zhang's inequalities fails at least for some GVFs.
    
    Let $K$ be the function field of the Jacobian of a curve of genus $g > 1$ endowed with the structure of a GVF by the $\theta$-polarization and let $\ov{K}$ be its algebraic closure with the symmetric extension GVF structure. Then, it follows from the counterexample that $\ov{K}$ is not an existentially closed GVF.
\end{remark}

\begin{proposition}[cf.\ Proposition 6.4.4 \cite{Adelic_curves_1}]\label{prop:zhang_inequality}Let $X$ be an integral projective scheme over $K$ and $\ov{\sL}$ a geometrically ample semipositive adelic line bundle on $X$. Then, the following Zhang inequality holds
    \[
        \essmin(\ov{\sL}) \geq h_{\ov{\sL}}(X).
    \]
\end{proposition}

\begin{proof}
    Assume without loss of generality that $F = K$ is algebraically closed. Fix an adelic curve structure on a finitely generated subfield $L$ over which $X$ and $\ov{\sL}$ are defined. By arithmetic Hilbert-Samuel over adelic curves, see \cite[Theorem B]{chenpositivity}, it follows that the maximal asymptotic slope is bigger than $h_{\ov{\sL}}(X)$. This puts us in the setting of Proposition 6.4.4 in \cite{Adelic_curves_1}. We give a partial account of the proof in order to verify that the statement extends to the setting of globally valued fields.
    
    One first reduces to proving that for every subspace $V \subset H^0(X,\ov{\sL})$ over $L$, one has $\adeg(V) \leq \essmin(\ov{\sL}) + \epsilon$. Let $\Lambda = \{x \in X(F)\ |\ h_{\ov{\sL}}(x) \leq \essmin(\ov{\sL}) + \epsilon\}$. Then, one may choose points $P_1,\dots,P_{\dim V}$ such that the evaluation map $f:V \otimes_{L} K \to \bigoplus_{i=1}^{\dim V} \kappa(P_i)$ is a bijection. We observe that the points $P_1,\dots,P_{\dim V}$ can be defined over a countable subfield $L'\supset L$ of $K$. One may hence choose a representing adelic curve $L'$ and proceed with the proof in \cite{Adelic_curves_1} verbatim.
\end{proof}

This can be used to prove a uniform Zhang inequality differing substantially from \cite[Theorem 1.1]{mavraki2025quantitativedynamicalzhangfundamental}.

\begin{corollary}\label{cor:uniformzhang}
    Let $\sX \to S$ be a family of projective varieties over $\Qbar$ and let $\ov{\sL}$ be an adelic line bundle on $\sX$ that is semipositive on fibers. Assume furthermore that the geometric Deligne pairing $\langle \sL, \dots,\sL\rangle$ is big on $S$. Then, for any $\epsilon > 0$ there exist constants $N$ and $C$ such that the set 
    \[
        \{x\in \sX(s)(\Qbar)\ |\ h_{\sL}(x) < (1-\epsilon) h_{\sL}(\sX(s))\}
    \]
    is contained in a proper subvariety of degree $\leq N$ and height $< C$ for all $s \in S(\Qbar)$ outside a proper closed subset of $S$.
\end{corollary}

\begin{proof}
    Suppose, to the contrary, that there exists a sequence $s_i \in S(\Qbar)$ for which the claim fails. The set of points $s_i[h_{\sL}(\sX(s_i))^{-1}] \in S(\Qbar[h_{\sL}(\sX(s_i))^{-1}])$ is of bounded height by the assumption on geometric Deligne pairing. For a non-principal ultrafilter, one obtains a GVF extension $F$ of $\prod_{\sU} \Qbar[h_{\sL}(\sX(s_i))^{-1}]$ and a limit point $s \in S(F)$ such that 
    \[
    \{x\in \sX(s)(F)\ |\ h_{\sL}(x) \leq (1-\epsilon) h_{\sL}(\sX(s)) = (1-\epsilon)\cdot 1\}
    \]
    is Zariski dense by Lemma \ref{lemm:parmin_ultraproducts}. This contradicts the inequality stated in Proposition \ref{prop:zhang_inequality}.
\end{proof}

\section{\label{sec:bogomolov}Bogomolov conjecture}

The Bogomolov conjecture for higher dimensional subvarieties was first proved over number fields in \cite{zhang_bogomolov}. The geometric case turned out to be substantially more difficult since the equidistribution method cannot be applied as easily in the geometric case. The geometric Bogomolov conjecture was proven in characteristic $0$ in \cite{cantat_gao_habegger_geometric_bogomolov} and in general characteristic in \cite{xie_yuan_geometric_bogomolov}. Subsequently, the geometric Bogomolov conjecture for semiabelian varieties was reduced to the corresponding result for abelian varieties in \cite{luo2025geometricbogomolovconjecturesemiabelian}. The Bogomolov conjecture was proven for abelian varieties over adelic curves with a positive measure of archimedean absolute values in \cite{Adelic_curves_4}.

The Bogomolov conjecture admits a formulation over arbitrary globally valued fields and all of the above results are instances of this extension. Its formulation requires additional care to accommodate the more general nature of GVFs.

\subsection{Preliminaries on semiabelian varieties}

In this section, we define semiabelian varieties and recall their basic properties. We define compactifications and  canonical heights of semiabelian varieties and lastly introduce the Chow trace to deal with isotrivial parts of semiabelian varieties.

\begin{definition}
    A semiabelian variety $G$ over a field $K$ is a group variety over $K$ that is the extension of an abelian variety by a torus. That is, there exists a short exact sequence of fppf-sheaves of abelian groups
    \[
        \begin{tikzcd}
            0 \arrow[r] & \bbT  \arrow[r, "i"] & G \arrow[r, "\pi"] & A \arrow[r] & 0.
        \end{tikzcd}
    \]
\end{definition}

This short exact sequence is determined up to unique isomorphism. We refer to $\bbT$ as the torus part of $G$ and to $A$ as the abelian quotient. After possibly replacing $K$ by a finite field extension, there exists an isomorphism $\bbT \cong \bbG_m^t$, where $t$ is referred to as the torus rank of $G$. The Weil-Barsotti formula states that the set of extensions of $A$ by $\bbG_m$ can be identified with $A^\vee(K)$. This identification is given by the fact that any semiabelian variety defines a $\bbT$-torsor over $A$ which is automatically numerically trivial. All of the above definitions and results carry over to semiabelian schemes over a base scheme $S$.

We will restrict ourselves to the case that the torus is split, i.e.\ there exists an isomorphism $\bbT \cong \bbG_m^t$ or $\bbT_S \cong \bbG_{m,S}^t$ in the relative setting. Note that over algebraically closed fields the torus is automatically split. We describe compactifications and the canonical height on a semiabelian variety following \cite{hultberg2024arakelovgeometrytoricbundles}. We do not include background on toric varieties and refer to loc.\ cit.\ for additional details. The reader may opt to skip the discussion in favor of a more ad hoc approach.

Let $\bbT$ be a split torus and $\sT \to B$ be a $\bbT$-torsor. Then, every character $\sT \to \bbG_m$ gives rise to a line bundle on $B$. Conversely, knowing the resulting line bundles allows us to recover the $\bbT$-torsor. Since we know how to define line bundles in the arithmetic setting, this allows us to define $\bbT$-torsors in the arithmetic setting as well. Let $M$ denote the characters of $\bbT$ and $N$ the co-characters of $\bbT$.

\begin{definition}
    A $\bbT$-bundle can be identified with a collection of line bundles $(\Tbun(m))_{m\in M}$ on $B$ indexed by $M$ with compatible identifications $\Tbun(m_1 + m_2) \cong \Tbun(m_1) \otimes \Tbun(m_2)$. By changing the notion of line bundles to metrized line bundles and requiring the identifications to be isometries, we can define metrized $\bbT$-bundles. In the same way one can define adelic $\bbT$-bundles.
\end{definition}

An integral polytope $\Delta \subset M_\bbR$ defines a compactification $X_\Delta$ of $\bbT$ as a projective toric variety together with an ample torus invariant Cartier divisor $D(\Delta)$. The Cartier divisor $D(\Delta)$ can be described in terms of its global sections. We note that every element $m \in M$ defines a meromorphic function $m$ on $X_\Delta$ by extending the character $m:\bbT \to \bbG_m$. Let $m_1, \dots, m_n$ be the vertices of $\Delta$. Then, $D(\Delta) = \bigvee_{i=1}^n \textrm{div}(m_i)$. We compactify a toric bundle $\Tbun$ over a base $B$ by $\toricbun_\Delta = (\Tbun \times \toricvar_\Delta)/\T$, where $x \in \T$ acts by $(x,x^{-1})$. Denote the map to the base by $\pi: \toricbun_\Delta \to B$. Since $D(\Delta)$ is torus invariant it follows that $\rho(D(\Delta)) = (\Tbun \times D(\Delta))/\T$ is a well-defined Cartier divisor on $\toricbun_\Delta$. The construction satisfies
\[
    \rho(D(\Delta)) = \bigvee_{i=1}^n \rho(\textrm{div}(m_i)).
\]
The line bundle induced by $\rho(\textrm{div}(m_i))$ is precisely $\pi^*\Tbun(m_i)$. 

The above discussion extends to the metrized setting. Let $K$ be a valued field and endow $D(\Delta)$ with the canonical metric to obtain a metrized toric divisor $\ov{D}(\Delta)$ and let $\aTbun$ be a metrized toric bundle. Then, there is a metric analogue $\arho$ of the construction $\rho$. It satisfies
\[
    \arho(\ov{D}(\Delta)) = \bigvee_{i=1}^n \arho(\adiv(m_i)).
\]
The metrized line bundle induced by $\arho(\adiv(m_i))$ is precisely $\pi^*\aTbun(m_i)$. The construction globalizes since $\arho(\adiv(m_i))$ defines an adelic divisor if $\aTbun$ is an adelic torus bundle and taking maxima preserves adelic line bundles. To shorten notation we will denote the (adelic) divisor on the toric bundle constructed above by $\rho(\Delta)$ or $\arho(\Delta)$ respectively. If $0 \in \Delta$, it follows that $\arho(\Delta)$ is effective. Its geometric vanishing locus is contained in the boundary of the toric bundle.

\begin{example}
    An identification $M \cong \bbZ^n$ induces an isomorphism $\bbT \cong \bbG_m^n$. The polytope $[-1,1]^n$ defines the compactification $(\bbP^1)^n$. The toric Cartier divisor defined by the data is precisely $\sum^n_{i=1} (\bbP^1)^{i-1}\times[0]\times(\bbP^1)^{n-1} + (\bbP^1)^{i-1}\times[\infty]\times(\bbP^1)^{n-1}$.
\end{example}

We define canonical heights on semiabelian varieties by adding an abelian and a toric contribution. Let $G$ be a semiabelian variety over a GVF $K$ with torus part $\bbT$ and abelian quotient $A$. Fixing an integral polytope $\Delta\subset M_\bbR$ determines a compactification $\ov{G} =G_{\Delta}$ of $G$. Multiplication by a natural number $n$ extends to a morphism $[n]:\ov{G} \to \ov{G}$. We may endow $\rho(\Delta)$ with unique metrics such that $[n]^*\rho(\Delta) = n \rho(\Delta)$ extends to an equality $[n]^*\ov{\rho(\Delta)} = n \ov{\rho(\Delta)}$.

The natural $\bbT$-bundle structure of $G$ extends to an adelic $\bbT$-bundle by endowing the numerically trivial line bundles on $A$ with their canonical metric. One obtains $\arho(\Delta) = \ov{\rho(\Delta)}$ since also $[n]^*\arho(\Delta) = n \arho(\Delta)$. If $0 \in \Delta$, then $h_{\arho(\Delta)} \geq 0$ on $G$, but not necessarily on the boundary of $\ov{G}$. By Lemma \ref{lemm:maxima_of_divisors_and_positivity}, $\arho(\Delta)$ is integrable if $\aTbun(m)$ is integrable for all $m$ and semipositive if $\aTbun(m)$ is flat for all $m$. Hence, $\arho(\Delta)$ is integrable and semipositive.

Let $\ov{\sM}$ be an ample symmetric line bundle endowed with the canonical metric on $A$ and $\Delta$ a polytope containing $0$ in its interior. We define the canonical height on $G$ by $\NTht = h_{\pi^*\ov{\sM}} + h_{\arho(\Delta)}$. The dependence of $\NTht$ on $\sM$ and $\Delta$ is mild. For two canonical heights $\NTht_1,\NTht_2$ constructed as above there exists a constant $C$ such that $C^{-1}\NTht_1\leq\NTht_2\leq C\NTht_1$ on $G$. For the abelian contribution this inequality is standard. For the toric contribution it suffices to observe that for any polytope $\Delta$ one has $h_{\arho(N\Delta)} = N h_{\arho(\Delta)}$ and for polytopes $\Delta_1 \subset \Delta_2$ one has $h_{\arho(\Delta_1)}\leq h_{\arho(\Delta_2)}$.

In order to deal with questions of isotriviality we introduce the Chow trace.

\begin{definition}
    Let $K/k$ be an extension of algebraically closed fields and $G$ be a semiabelian variety over $K$. The Chow trace $(G^{K/k},\tr)$ of $G$ is defined as the final object in the category of pairs $(G_0,h)$ of a semiabelian variety $G_0$ defined over $k$ together with a map $G_0 \otimes_k K \to G$.
\end{definition}

It is well-defined by \cite[Corollary 2.4.4 (2)]{liu2024chowtrace1motiveslangneron} and $\tr$ has trivial kernel. For more details in the abelian setting, we refer to \cite[\S 6]{conrad_lang-neron}.

\subsection{Statements and equivalent formulations}

Let $A$ be an abelian variety over an algebraically closed GVF $K$. If $K$ is non-archimedean let $k$ be its constant field and denote by $(A^{K/k},\tr)$ its $K/k$-Chow trace. Let $\sM$ be an ample symmetric line bundle on $A$ and let $\widehat{h}$ denote the associated N\'eron-Tate height.

A closed integral subvariety $X$ of $A$ is called \emph{special} if it is of the form
\[
    \tr(Y\otimes K) + V + a,
\]
where $Y$ is a subvariety of $A^{K/k}$, $V$ is an abelian subvariety of $A$ and $a\in A(K)$ is a point satisfying $\widehat{h}(a) = 0$.

\begin{conjecture}[Bogomolov conjecture]
    Any subvariety $X \subset A$ such that for every $\epsilon > 0$ the set $\{x\in X(K)\ |\ \NTht(x) < \epsilon\}$ is Zariski dense is special. This condition is moreover equivalent to $\NTht(X) = 0$.
\end{conjecture}

It follows from Proposition \ref{prop:zhang_inequality} that any subvariety $X$ with a Zariski-dense set of small points satisfies $\NTht(X) = 0$. It follows conversely that if $\NTht(X) = 0$ there exists some GVF extension $F/K$ with a Zariski-dense set of small points or even height $0$ points by Proposition \ref{prop:lower_zhang_ineq} and Remark \ref{rk:successive_minima_attained}. Using Poincar\'e reducibility this will later show the equivalence of the height and the small points formulation of the Bogomolov conjecture.

It is necessary to consider height $0$ points that are neither torsion or defined over the field of constants.

\begin{example}
    Let $A$ be an abelian variety over $K$. Then, the diagonal in $A \times A$ defines a height $0$ point in $A_{K(A)}$ with $K(A)$ endowed with the canonical height.
\end{example}

\begin{example}
    Let $\sE \to C$ be a non-isotrivial elliptic surface over a GVF $K$ of characteristic $0$. Let $\ov{\sM}$ be a canonically metrized ample symmetric line bundle. Let $P\in \sE(K(C))$ be a non-torsion section. Then, by \cite{gao_generic_rank_betti}, one sees that the underlying geometric adelic line bundle $\sM|_{\ov{P}}$ is big. By \cite[Proposition 1.5]{demarco_bifurcation_intersections_heights} there exist infinitely many $t\in C(K)$ such that $P_t$ is torsion. Using ultraproducts, one may define a GVF structure on $K(C)$ extending the GVF structure on $K$ such that $\widehat{h}(P) = 0$. 
\end{example}

The Bogomolov conjecture can be formulated more generally for semiabelian varieties. Let $G$ be a semiabelian variety over an algebraically closed GVF $K$ with abelian quotient $\pi: G \to A$. If $K$ is non-archimedean let $k$ be its constant field and $(G^{K/k},\tr)$ its $K/k$-Chow trace.

A closed integral subvariety $X$ of $G$ is called \emph{special} if 
\[
    \widetilde X:=X/\mathrm{Stab}(X)=\tr(X_0\otimes K)  + g,
\]
where $X_0$ is a subvariety of $\widetilde G^{K/k}$ and $g\in \widetilde G(F)$ satisfying $\widehat{h}(g) = 0$, where $\widetilde G:=G/\mathrm{Stab}(X)$ and $F$ is a GVF extension of $K$.

\begin{remark}
    We are not aware of any example in which it is impossible to find $g\in \widetilde G(K)$ of height $0$ in the definition of special subvariety.
\end{remark}

Fix a canonically metrized ample symmetric line bundle $\ov{\sM}$ on the abelian quotient $A$ with underlying line bundle $M$ and an integral polytope $\Delta \subset \bbR^n$ containing the origin in its interior. The height and successive minima of a line bundle restricted to a subvariety $X \subset G$ will be understood to refer to the corresponding quantities for the closure $\ov{X}$ in the equivariant compactification $\ov{G}$ associated to $\Delta$.

\begin{conjecture}[Bogomolov conjecture]
    Any subvariety $X \subset G$ such that
    \[
    \adeg(\ov{\sM}^{\dim \pi(X) + 1}|\pi(X)) = 0
    \]
    and
    \[
    \adeg(\arho(\Delta)^{\dim X +1-\dim \pi(X)}\pi^*\ov{\sM}^{\dim \pi(X)}|X) = 0
    \]
    is special. This is moreover equivalent to the existence of a GVF extension $F$ such that for every $\epsilon > 0$ the set $\{x\in X(F)\ |\ \NTht(x) < \epsilon\}$ is Zariski dense.
\end{conjecture}

\begin{remark}
    When defining the canonical height on a semiabelian variety one encounters the problem that a semiabelian variety does not define a polarized dynamical system. Instead, the canonical height is defined as a sum $h_{\pi^*\ov{\sM}} + h_{\arho(\Delta)}$. With respect to pullback by $[n]$, the line bundle $\pi^*\ov{\sM}$ is in the eigenspace associated to $n^2$ and $\arho(\Delta)$ is in the eigenspace associated to $n$. The contribution of $\arho(\Delta)$ should therefore intuitively be infinitesimally small compared to $\pi^*\ov{\sM}$. The two intersection numbers in the conjecture are the first two terms in $(\pi^*\ov{\sM} + \arho(\Delta))^{\dim X}$. Fittingly, the remaining terms in the expansion are infinitesimally smaller and do not play a role.
\end{remark}

\begin{lemma}
    Any subvariety $X \subset G$ such that for every $\epsilon > 0$ the set $\{g\in G(K)\ |\ \NTht(x) < \epsilon\}$ is Zariski dense satisfies the following two equalities of intersection numbers
    \[
    \adeg(\ov{\sM}^{\dim \pi(X) + 1}|\pi(X)) = 0
    \]
    and
    \[
    \adeg(\arho(\Delta)^{\dim X +1-\dim \pi(X)}\pi^*\ov{\sM}^{\dim \pi(X)}|X) = 0.
    \]
\end{lemma}

\begin{proof}
    Let $X$ be a subvariety $X \subset G$ such that for every $\epsilon > 0$ the set $\{g\in G(K)\ |\ \NTht(x) < \epsilon\}$ is Zariski dense.
    
    We first observe that a Zariski dense set of small points in $X$ under $\pi$ maps to a Zariski dense set of small points in $\pi(X)$. Hence, the first equality follows from \ref{prop:zhang_inequality}. By the projection formula we obtain 
    \[
        (\arho(\Delta) + \pi^*\ov{\sM})^{\dim X + 1} = \sum_{i=0}^{\dim \pi(X)+1} \binom{\dim(X)+1}{i} \pi^*\ov{\sM}^i \arho(\Delta)^{\dim X +1-i}.
    \]
    The summand $\pi^*\ov{\sM}^{\dim \pi(X) + 1} \arho(\Delta)^{\dim X - \dim \pi(X)}$ vanishes as it has to equal
    \[
        \adeg(\ov{\sM}^{\dim \pi(X) + 1}|\pi(X))\cdot\deg((\rho(\Delta))^t)=0
    \]
    by the projection formula. The property of having a Zariski dense subset of small points is preserved under scaling of $\arho(\Delta)$ and $\ov{\sM}$. We use this to obtain the following series of inequalities
    \begin{align*}
        &\essmin(\arho(\Delta) + \pi^*\ov{\sM}^{\otimes N})\\ &\geq (\arho(\Delta) + \pi^*\ov{\sM}^{\otimes N})^{\dim X + 1}\\
        &=  N^{\dim \pi(X)}\binom{\dim(X)+1}{\dim\pi(X)} \pi^*\ov{\sM}^{\dim \pi(X)}\arho(\Delta)^{\dim X +1-\dim \pi(X)} + O(N^{\dim \pi(X)-1}).
    \end{align*}
    It follows that
    \[
        \adeg(\arho(\Delta)^{\dim X +1-\dim \pi(X)}\pi^*\ov{\sM}^{\dim \pi(X)}|X) \leq 0.
    \]
    It remains to prove the reverse inequality. For this we note that both $\arho(\Delta)$ and $\pi^*\ov{\sM}$ are semipositive and that $\pi^*\ov{\sM}$ is nef. In particular, we have $\absmin(\arho(\Delta) + \pi^*\ov{\sM}^{\otimes N}) \geq \absmin(\arho(\Delta) + \pi^*\ov{\sM})$. By the lower Zhang inequality for GVFs, see Proposition \ref{prop:lower_zhang_ineq}, we obtain that
    \begin{align*}
        &(\arho(\Delta) + \pi^*\ov{\sM}^{\otimes N})^{\dim X + 1}\\
        &=  N^{\dim \pi(X)}\binom{\dim(X)+1}{\dim\pi(X)} \pi^*\ov{\sM}^{\dim \pi(X)}\arho(\Delta)^{\dim X +1-\dim \pi(X)} + O(N^{\dim \pi(X)-1})\\
        &\geq \absmin(\arho(\Delta) + \pi^*\ov{\sM}^{\otimes N})\\
        &\geq \absmin(\arho(\Delta) + \pi^*\ov{\sM}).
    \end{align*}
    This concludes the proof that
    \[
        \adeg(\arho(\Delta)^{\dim X +1-\dim \pi(X)}\pi^*\ov{\sM}^{\dim \pi(X)}|X) = 0.
    \]
\end{proof}
\begin{lemma}\label{lemm:semiab_intersection_small_implies_small_points}
    Let $X$ be a subvariety of a semiabelian variety $G$ satisfying 
    \[
    \adeg(\ov{\sM}^{\dim \pi(X) + 1}|\pi(X)) = 0
    \]
    and
    \[
    \adeg(\arho(\Delta)^{\dim X +1-\dim \pi(X)}\pi^*\ov{\sM}^{\dim \pi(X)}|X) = 0.
    \]
    Then, there exists a GVF extension $F$ such that 
    \[
        \{x\in X(F)\ |\ \NTht(x) < \epsilon\}
    \]
    is Zariski dense for every $\epsilon > 0$. In fact, there exists a GVF extension $F$ such that
    \[
        \{x\in X(F)\ |\ \NTht(x) = 0\}
    \]
    is Zariski-dense.
\end{lemma}

\begin{proof}
    Let $\ov{\sL} = \ov{\sM} + \arho(\Delta)$. Let us study \[\lim_{n \to \infty} \frac{1}{n} h_{[n]^*\ov{\sL}}(X) = \lim_{n \to \infty} \frac{(n^2\ov{\sM} + n \arho(\Delta))^{\dim X + 1}}{n\cdot(\dim X +1)\cdot (n^2\sM + n \rho(\Delta))^{\dim X}}.\] By assumption on the intersection numbers the numerator is in $O(n^{\dim \pi(X) + \dim X})$. On the other hand the numerator has order of growth $n^{\dim \pi(X) + \dim X + 1}$. Hence $\lim_{n \to \infty} \frac{1}{n} h_{[n]^*\ov{\sL}}(X) = 0$.
    
    We observe that $\essmin([n]^*\ov{\sL}|_X) \geq n \essmin(\ov{\sL}|_X) \geq 0$ and that $\absmin([n]^*\ov{\sL}|_X) \geq \absmin(\ov{\sL}|_{\ov{G}}) > - \infty$. Hence, $\lim_{n \to \infty} \frac{1}{n} \absmin([n]^*\ov{\sL}|_X) = 0$. By the lower Zhang inequality in Proposition \ref{prop:lower_zhang_ineq} applied to $[n]^*\ov{\sL}|_X$ it follows that $\lim_{n \to \infty} \frac{1}{n} \essmin([n]^*\ov{\sL}|_X) = 0$ and thus $\essmin(\ov{\sL}|_X) = 0$.
\end{proof}

\begin{remark}
    The above proof replaces the reliance on an alternate version of the fundamental inequality in \cite[Theorem 2.4]{luo2025geometricbogomolovconjecturesemiabelian} requiring a new invariant called the sectional minimum by a simpler method at the expense of having to pass to a GVF extension. We believe that their argument applies also for general GVFs allowing to state the Bogomolov conjecture without the need of extending the GVF. We choose our approach to illustrate the new methods that are available by allowing GVF extensions.
\end{remark}
    
\section{\label{sec:bogomolov_gvf}New gap principle}

After introducing moduli of abelian and semiabelian varieties we proceed to prove the deduction of the new gap principle from the Bogomolov conjecture for GVFs in both the abelian and semiabelian setting.

\subsection{Moduli of semiabelian varieties}

In order to deduce uniform results for all subvarieties of all semiabelian varieties it is necessary to put them in finitely many families. Analogous arguments are present in \cite{gao_ge_kuehne_uniform_mordell_lang} in the case of abelian varieties. 

We introduce the spaces parametrizing abelian and later semiabelian varieties.

\begin{theorem}[Theorem 7.9 \cite{mumford_fogarty_kirwan_GIT}]
    For every $g,d$ and sufficiently big $N$, there exists a fine moduli space of abelian varieties of dimension $g$, with polarization of degree $d$ and level $N$ over $\Spec \bbZ[\frac{1}{N}]$. It is denoted by $\abmoduli_{g,d,N}$ and is quasi-projective over $\Spec \bbZ[\frac{1}{N}]$.\footnote{In loc.cit. it is claimed that this holds over $\Spec \bbZ$.}
\end{theorem}

Denote by $\univabvar_{g,d,N}$ the universal family of abelian varieties. We denote by $\abmoduli$ the disjoint union of $\abmoduli_{g,d,N}$ over varying dimension, polarization degree and level structure and by $\univabvar$ the universal abelian variety over $\abmoduli$. We write $\abmoduli_g$ and $\univabvar_g$ to denote the locus of abelian varieties of dimension $g$. We need to consider non-principal polarizations in order to work in positive characteristic.

Following \cite[\S 6.2]{mumford_fogarty_kirwan_GIT} one can define a symmetric relatively ample line on $\univabvar$ inducing twice the polarization. Using \cite[\S 6.1]{yuan_zhang_adeliclinebundlesquasiprojective}, we enhance it to an adelic line bundle $\tautlb$ inducing the fibrewise N\'eron-Tate height.

We denote the relative restricted Hilbert space parametrizing subvarieties of the universal abelian variety of dimension $k$ and degree $l$ by $\univsubvarofab_{k,l} = \coprod\resHilbert_P(\univabvar/\abmoduli)$ where $P$ goes over the possible Hilbert polynomials.

\begin{definition}
    We define the moduli space of semiabelian varieties of torus dimension $t$, abelian dimension $g$ and level $N$ structure and polarization of degree $d$ as the t-fold fibre product the of the dual of the universal abelian variety
    \[
        \semiabmoduli_{g,t,d,N} = (\univabvar^\vee_{g,d,N})^{\times t}.
    \]
\end{definition}

This parametrizes (split) semiabelian varieties with identification of the torus part $\bbT \cong \bbG^t_m$. This follows by the Barsotti-Tate formula that identifies $\Ext^1(A, \bbG_m) \cong A^\vee$ for abelian schemes $A$. 

The universal semiabelian variety $\univsemiab_{g,t,d,N}$ is the universal $\bbG^t_m$-bundle over $\univabvar_{g,d,N} \times(\univabvar^\vee_{g,d,N})^{\times t}$. By choosing a trivialization of the $\bbG^t_m$-bundle $\univsemiab_{g,t,d,N}$ over $0 \times (\univabvar^\vee_{g,d,N})^{\times t}$, where $0$ denotes the $0$-section of $\univabvar_{g,d,N}$, we endow $\univsemiab_{g,t,d,N}$ with a group structure.

Fix an integral polytope $\Delta \subset \bbR^n$ containing the origin in its interior. Then, this defines a relative compatification of $\univsemiab$ over $\univabvar$. The line bundle $\rho(\Delta) + \sM$, where $\sM$ is an ample symmetric line bundle on the $\univabvar$ inducing twice the polarization, is a relatively ample line bundle for $\univsemiab \to \semiabmoduli$.

We denote the relative restricted Hilbert space parametrizing subvarieties of degree $l$ in the universal semiabelian variety of abelian dimension $g$ and torus dimension $t$ by $\univsubvarofab_{g,t,l} = \coprod\resHilbert_P(\univabvar/\abmoduli)$ where $P$ goes over the possible Hilbert polynomials.

\subsection{Reduction to the Bogomolov conjecture}

Let $A$ be a polarized abelian variety over a GVF $K$ and let $X \subset A$ be a closed subvariety with finite stabilizer such that $X - X$ generates $A$. The stable Faltings height of an abelian variety $A$ will be denoted $\Falht(A)$. In this article, Faltings height will always refer to the stable Faltings height. It is induced by the Hodge line bundle on the moduli space of abelian varieties.

\begin{theorem}[New gap principle over $\Qbar$, Theorem 1.2 in \cite{gao_ge_kuehne_uniform_mordell_lang}]\label{thm:ngp_qbar}
    There exist positive constants $c_1 = c_1(\dim A, \deg_L X)$ and $c_2 = c_2(\dim A, \deg_L X)$ with the property
    \[
        \left\{P \in X(\Qbar)\ |\ \NTht(P) \leq c_1 \max\{1,\Falht(A)\}\right\}
    \]
    is contained in some Zariski closed $X' \subsetneq X$ with $\deg_L(X') < c_2$.
\end{theorem}

The quantity $\max\{1,\Falht(A)\}$ is comparable with the height with respect to any arithmetically ample line bundle on the moduli space of abelian varieties. The above formulation does not hold verbatim over $\ktbar$ as $\max\{1,\Falht(A)\}$ is bigger than the height with respect to an ample line bundle. 

\begin{theorem}[\cite{moret-bailly_pinceaux_varietes_abeliennes} IX Theor\`eme 3.2]
    The Hodge line bundle on the compactification of $\abmoduli_{g,d,N}$ is ample.
\end{theorem}

Let $K$ be a non-archimedean globally valued field with constant field $k$. Then, the above theorem implies that for an abelian variety $A$ over $K$ one has $\Falht(A) \geq 0$ with equality if and only if $A$ can be defined over $k$. In this case we call $A$ \emph{isotrivial}.

Let $G$ be a semiabelian variety over a GVF $K$ together with an identification of its torus part $\bbT \cong \bbG_m^t$ and an ample symmetric line bundle $M$ on its abelian quotient $A$. Due to the Barsotti-Weil formula, the semiabelian variety $G$ gives rise to points $Q_1, \dots, Q_n \in A^\vee(K)$. We associate to $G$ the height $h = \Falht + \NTht_{\textrm{NT}}(Q_1) + \dots+\NTht_{\textrm{NT}}(Q_t)$, where $\NTht_{\textrm{NT}}$ denotes the N\'eron-Tate height on $A^\vee$. We observe that $G$ is isotrivial if and only if $\max\{\height(2),h(G)\} =0$. Fix an integral polytope $\Delta \subset \bbR^n$ containing the origin in its interior. The above data determine a choice of canonical height $\NTht = h_{\pi^*\ov{\sM} + \arho(\Delta)}$ on $G$. The most general form of the new gap principle is the following.

\begin{conjecture}[New gap principle for semiabelian varieties]
    There exist positive constants $c_1 = c_1(\dim G, \deg_{\pi^*M + \rho(\Delta)} X)$ and $c_2 = c_2(\dim G, \deg_{\pi^*M + \rho(\Delta)} X)$ such that
    \[
        \left\{P \in X(K)\ |\ \NTht(P) < c_1 \max\{\height(2),h(G)\}\right\}
    \]
    is contained in a proper Zariski closed subset $X' \subsetneq X$ with $\deg_{\pi^*M + \rho(\Delta)}(X') < c_2$ for all semiabelian varieties $G$ over a GVF $K$ and closed subvarieties $X \subset G$ with finite stabilizer such that $X - X$ generates $G$.
\end{conjecture}

Our goal for the section is to prove the following Proposition.

\bogimpliesngp*

Theorem \ref{thm:ngpqbar}, \ref{thm:ngpbigchar} and \ref{thm:ngpellipticquotient} will follow immediately from the instances of the Bogomolov conjecture proved in Theorem \ref{thm:bcsa}. More precisely, Theorem \ref{thm:ngpqbar} requires the Bogomolov conjecture for all characteristic $0$ GVFs, Theorem \ref{thm:ngpbigchar} only requires the Bogomolov conjecture for all non-archimedean characteristic $0$ GVFs and Theorem \ref{thm:ngpellipticquotient} only requires the Bogomolov conjecture for semiabelian varieties whose quotient is an elliptic curve.

\begin{proof}
The general strategy is to first parametrize subvarieties of semiabelian varieties. Then, we take for any $c_1$ and $c_2$ a counterexample to the statement of the new gap principle. By taking a limit point of these counterexamples valued in an ultraproduct of their fields of definition. For this limit point the Bogomolov conjecture would have to fail. The application to Theorem \ref{thm:ngpqbar} uses that the ultraproduct of GVFs of the form $\Qbar[\lambda]$ is of characteristic $0$ and the application to Theorem \ref{thm:ngpbigchar} uses that the ultraproduct of GVFs whose characteristic is $0$ or tends to $\infty$ is a GVF of characteristic $0$. Theorem \ref{thm:ngpellipticquotient} is obtained by using the moduli space of semiabelian varieties with elliptic quotient instead of the moduli space of all semiabelian varieties.

Any pair $X \subset G$ consisting of a subvariety in a semiabelian variety over a GVF $K$ together with an identification of its torus part $\bbT \cong \bbG_m^t$ and an ample symmetric line bundle $M$ on its abelian quotient $A$ defines a $K$-valued point of the appropriate moduli space $\univsubvarofab_{g,t,l}$ after fixing some level structure coprime to the characteristic of $K$ on the abelian quotient. Refer to a semiabelian variety with the mentioned extra data as an \emph{adorned semiabelian variety}.

In order to reduce the new gap principle to finitely many families the polarization degree of the abelian quotient, or in other words $\deg_M(A)$, needs to be bounded in terms of the degree of the subvariety $X$. We note that by the nefness of the involved line bundles we have the following inequalities of intersection numbers on $X$, where $\eta$ denotes the generic point of $\pi(X)$.
    \begin{align*}
        (\rho(\Delta) + M)^{\dim(X)}_X &\geq \rho(\Delta)^{\dim (X) - \dim \pi (X)} M^{\dim \pi (X)}\\ &= \deg_{\rho(\Delta)}(X_\eta) \cdot \deg_M(\pi(X)) \geq \deg_M(\pi(X))
    \end{align*}
The bound on the polarization degree follows from \cite[Lemma 2.5]{gao_ge_kuehne_uniform_mordell_lang}.

We replace the Faltings height by a more positive height to strengthen the statement and obtain a uniform statement of the new gap principle for $\Qbar$ and non-archimedean GVFs.

Let $T \subset \univsubvarofab$ be a closed irreducible subvariety of the Hilbert scheme. Let $\ov{\sL}$ be an arithmetically ample adelic line bundle on a compactification of $T$. By arithmetic ampleness, there exists a constant $C$ such that
    \[
        \max\{\height(2),h(\univsemiab_t)\} \leq C h_{\ov{\sL}}(t).
    \]

By noetherian induction we are restricted to proving the following statement. 

\begin{statement}
    Let $\sX \subset \sG$ be a horizontal subvariety of an adorned semiabelian scheme defined by a closed irreducible subvariety $T \subset \univsubvarofab$ such that the generic fibre has finite stabilizer and generates $\sG_\eta$. There exist numbers $\epsilon>0$ and $N$ such that
    \[
        \left\{x \in \sG_t(F)\ |\ \NTht(x) \leq \epsilon h_{\ov{\sL}}(t)\right\}
    \]
    is contained in a proper Zariski closed subset $Z_t$ of degree $\leq N$ for all GVF valued points $t \in T(F)$ in a Zariski dense open subset.
\end{statement}

We now prove the statement assuming the Bogomolov conjecture. Suppose by contradiction that for all reals $\epsilon, N >0$ that there is a Zariski dense set of GVF valued points $t\in T(K)$ such that
    \[
        \{x \in \sX_t(K)\ |\ \widehat{h}(x) \leq \epsilon h_{\ov{\sL}}(t) \}
    \]
is not contained in any proper Zariski closed subset of degree and height $\leq C$. Denote the set of such $t \in T(K)$ for fixed $\epsilon$ and $C$ by $T_{\epsilon, C}$. Let $I = \bigsqcup^\infty_{k=1} T_{\frac{1}{k},k}$ and let $\sU$ be an ultrafilter such that a subset $J \subset I$ is small if it is not Zariski dense or if $J = \bigsqcup^N_{k=1} T_{\frac{1}{k},k}$ for some $N$. Denote the ultraproduct by $F = \prod_{\mathcal U} K[h_{\ov{\sL}}(t_i)^{-1}]$. By assumption, $\sX_F$ has finite stabilizer. The limit of the $t_i$ defines a point $t \in T(F)$ satisfying $h_{\ov{\sL}}(t)=1$. Hence, we obtain that $\sX_F \subset \sG_F$ is not isotrivial and since $\sX_F$ generates $\sG_F$ it is not special. By the Bogomolov conjecture this implies that there exists an $\epsilon$ and a closed subvariety $Z \subset \sX_F$ such that all $x \in \sX_F(F)$ with $\NTht(x) \leq \epsilon$ are contained in $Z$. Applying Lemma \ref{lemm:parmin_ultraproducts} gives rise to a contradiction.
\end{proof}

\section{\label{sec:abelian_varieties_and_tori}Proof of the Bogomolov conjecture for semiabelian varieties}

This section is devoted to prove the Bogomolov conjecture over globally valued fields in characteristic $0$. The basic strategy to prove the Bogomolov conjecture for a semiabelian variety by reducing it to the Bogomolov conjecture for its abelian quotient is due to \cite{luo2025geometricbogomolovconjecturesemiabelian} in the case of polarized function fields. We prove such a reduction step over GVFs of any characteristic. The Bogomolov conjecture for abelian varieties over positive characteristic GVFs, however, remains open. The main additional difficulty of the reduction is the failure of the dynamical Northcott property, i.e.\ for a traceless abelian variety $A$ over a GVF $K$ there may exist non-torsion points of height $0$.

We first prove the Bogomolov conjecture for quasi-split semiabelian varieties in section \ref{sec:quasi-split} by reducing it to the case of (traceless) abelian varieties. 

\begin{definition}
    A semiabelian variety $G$ over a GVF $K$ is called \emph{quasi-split}\footnote{This follows the terminology of \cite{luo2025geometricbogomolovconjecturesemiabelian}. There is no relationship to the notion of quasi-split reductive group or to the torus being split.} if there exists a semiabelian variety $G_0$ over $k$ and a traceless abelian variety $A$ such that $G = G_{0,K} \times A$. If $K$ is archimedean we take this to mean that $G = \bbG_m^t \times A$.
\end{definition}
 
This relies on the Bogomolov conjecture for (traceless) abelian varieties and tori that are not difficult to deduce from results in the literature. We discuss the Bogomolov conjecture on abelian varieties and tori over archimedean GVFs in section \ref{sec:archimedean} after giving a reminder on the Ueno locus in section \ref{sec:ueno}. In section \ref{sec:geometric} we treat the case of geometric GVFs. Lastly, we reduce the Bogomolov conjecture for arbitrary semiabelian varieties to the quasi-split case in section \ref{sec:reduction_to_quasisplit}.

We start out by proving some standard reductions that easily extend to the setting of GVFs.

\begin{lemma}\label{lemm:translation_preserves_smallness}
    The translate $g+X$ of a small subvariety $X$ by a point $g \in G(K)$ of height $0$ is again small.
\end{lemma}

\begin{proof}
    The pullback of the line bundles $\ov{\sM}$ and $\arho(\Delta)$ on a semiabelian variety $G$ along the addition map $G \times G$ are controlled by the box products $\ov{\sM} \boxtimes\ov{\sM}$ and $\arho(\Delta) \boxtimes\arho(\Delta)$. It follows that there exists a real number $C$ such that $\NTht(x+y) \leq C(\NTht(x) + \NTht(y))$. In particular, the lemma follows.
\end{proof}

\begin{lemma}[Proposition 3.2 \cite{luo2025geometricbogomolovconjecturesemiabelian}]\label{lemm:special_under_isogeny}
    The image $Y \subset G'$ of a subvariety $X\subset G$ under isogeny $G \to G'$ is special if and only if $X$ is.
\end{lemma}

\begin{proof}
    Quotienting out by the respective stabilizers of $X$ and $Y$ reduces the statement to the case that the respective stabilizers are finite. Isogenies preserve points of height $0$ and isotrivial subvarieties, implying that if $X$ is special so is $Y$. One can on the other hand find an isogeny $G' \to G$ such that the composition equals multiplication by $n$ for suitable $n$. Now it suffices to see that if $n X$ is special so is $X$ which is immediate.
\end{proof}

\begin{lemma}[Lemma 3.3 \cite{luo2025geometricbogomolovconjecturesemiabelian}]\label{lemm:small_under_isogeny}
    The image $Y$ of a subvariety $X$ under isogeny is small if and only if $X$ is.
\end{lemma}

\begin{proof}
    A homomorphism of semiabelian varieties $\phi:G \to G'$ induces a morphism of torus bundles. In particular, the data of an isogeny of abelian varieties consists of an isogeny of the abelian quotients $\phi_{ab}:A \to A'$ and the torus part $\phi_{tor}:\bbT \to \bbT'$. This induces a map $\widehat{\phi}:M' = \Hom(\bbT',\bbG_m) \to M = \Hom(\bbT,\bbG_m)$. Let $\Delta' \subset M'_\bbR$ be a polytope and $\Delta$ its image under $\widehat{\phi} \otimes \bbR$. Then, $\Delta$ and $\Delta'$ determine compactifications $\ov{G}$ and $\ov{G}'$ and $\phi$ extends to a morphism $\ov{\phi}:\ov{G} \to \ov{G}'$. Denote the projections to the abelian quotients by $\pi$ and $\pi'$ respectively. From an ample symmetric canonically metrized line bundle $\ov{\sM}'$ on $A'$ one can obtain an ample symmetric line bundle $\ov{\sM}=\phi_{ab}^*\ov{\sM}'$ on $A$. Then, $\ov{\phi}^* (\pi')^* \ov{\sM}' = \pi^* \sM$ and $\ov{\phi}^* \arho(\Delta') = \arho(\Delta)$. In other words, the pullback of a canonical height on $G'$ defines a canonical height on $G$. Since being small does not depend on the choice of canonical heights we conclude the result.
\end{proof}

\begin{lemma}\label{lemm:finite_stabilizer}
    The Bogomolov conjecture holds if and only if it holds for subvarieties with finite stabilizer.
\end{lemma}

\begin{proof}
    Let $X$ be a small subvariety of a semiabelian variety $G$. Let $G(X)$ be the connected component of the stabilizer of $X$. The image $\varpi(X)$ under the morphism $\varpi:G \to G'=G/G(X)$ is also small. In the formulation of the Bogomolov conjecture for semiabelian varieties it is clear that $X$ is special if $\varpi(X)$ is. The deduction for the standard formulation of the Bogomolov conjecture for abelian varieties relies on Poincar\'e reducibility and can be found in the last paragraph of the proof of \cite[Theorem 9.20]{moriwaki_arakelov_geometry}.
\end{proof}

\begin{lemma}\label{lemm:generates_abelian_variety}
    It suffices to prove the Bogomolov conjecture for small subvarieties $X$ that generate the semiabelian variety $G$.
\end{lemma}

\begin{proof}
    If $X$ is a small subvariety, then up to passing to a GVF extension $F$ it contains a point $g \in X(F)$ whose canonical height is $0$. By Lemma \ref{lemm:translation_preserves_smallness}, $X-g$ is again small. It is a special subvariety in the semiabelian subvariety generated by $X-g$ if and only if it is a special subvariety of $G$.
\end{proof}

\subsection{\label{sec:ueno}Ueno locus and grouplessness}

Instead of rewriting the proof of the Bogomolov conjecture for abelian varieties over archimedean GVFs we reduce it to the case of adelic curves proved in \cite{Adelic_curves_4}. This requires us to understand the behaviour of the Ueno locus under extensions of algebraically closed fields.\footnote{I learned the results in this section from Aryan Javanpeykar's answer on https://mathoverflow.net/questions/356197/translates-of-abelian-subvarieties}

\begin{definition}
    The \emph{Ueno} or \emph{Kawamata locus} of a closed subvariety $X \subset G$ of a finite type connected group schemes $G$ is the union $Z(X)$ of all translates of subgroup schemes of $G$ or in other words
    \[
        Z(X) = \{ x \in X\ |\ \exists B\subset G \textrm{ subgroup scheme, } \dim(B) > 0 \textrm{ s.th. } x + B \subset X\}.
    \]
\end{definition}

Its basic properties were studied in \cite{abramovich_subvarieties_of_semiabelian}. Its main result can be summarized as follows.

\begin{theorem}[\cite{abramovich_subvarieties_of_semiabelian}]
    Let $G$ be a semiabelian variety and let $X \subset G$ be a closed subvariety. Then,  the Ueno locus $Z(X)$ is a closed subvariety of $X$. It satisfies $Z(X) = X$ if and only if $\dim(\Stab(X))>0$.
\end{theorem}

We follow \cite[Section 3]{javanpeykar_xie_finiteness_properties_pseudo_hyperbolic} in introducing grouplessness which has been studied under extension of algebraically closed fields.

\begin{definition}[Pseudo-grouplessness]
    Let $X$ be a variety over an algebraically closed field $K$ and $\Delta \subset X$ a closed subvariety. We say that $X$ is \emph{groupless modulo $\Delta$} if for all finite type connected group schemes $G$ and every  Zariski dense open $U \subset G$ such that $\codim(G\setminus U) \geq 2$ every morphism $U \to X$ factors over $\Delta$.
\end{definition}

\begin{proposition}[Proposition 3.7 \cite{javanpeykar_xie_finiteness_properties_pseudo_hyperbolic}]
Let $K \subset F$ be an extension of algebraically closed field of characteristic $0$. Then, $X_F$ is groupless modulo $\Delta_F$ if and only if $X$ is groupless modulo $\Delta$.
\end{proposition}

In the case of semiabelian varieties the resulting groupless-exceptional locus agrees with the Ueno locus.

\begin{theorem}[Theorem 5.2.5 \cite{morrow_vojta_green-griffiths-lang-vojta-conjecture}]
     Let $G$ be a semiabelian variety and let $X \subset G$ be a closed subvariety. Then, $X$ is groupless modulo $Z(X)$.
\end{theorem}

\begin{corollary}\label{cor:stability_Ueno_locus}
    Let $G$ be a semiabelian variety and let $X \subset G$ be a closed subvariety. Let $K \subset F$ be an extension of algebraically closed fields. Then, the Ueno locus satisfies $Z(X_F) = Z(X)_F$.
\end{corollary}

\subsection{\label{sec:archimedean}Archimedean GVFs}

The Bogomolov conjecture for abelian varieties has been proven by Chen and Moriwaki in \cite{Chen_differentiability_of_arithmetic_volume} over archimedean adelic curves. A priori, only countable GVFs are represented by an adelic curve. In this section, we reduce to the case of countable GVFs and prove the Bogomolov conjecture for tori over archimedean GVFs.

\begin{lemma}
    The Bogomolov conjecture for abelian varieties over countable archimedean GVFs implies the Bogomolov conjecture over all archimedean GVFs.
\end{lemma}

\begin{proof}
    By Lemma \ref{lemm:finite_stabilizer}, we are reduced to the setting where the small subvariety $X$ has finite stabilizer. In particular, the Ueno locus $X^\circ$ is not empty. By assumption, there is an infinite number of small points in $X^\circ(K)$. Take a countable small sequence of distinct points $x_1,\dots\in X^\circ(K)$. It is possible to define $A$ and $X$ over a finite type field $M$ over $\bbQ$. The field $L$ generated by $x_1,\dots\in X^\circ(K)$ over $M$ is still countable. The Zariski closure of $\{x_1,x_2,\dots\}$ in $X$ contains a positive-dimensional irreducible component $Z$ which is small. Since $Z$ is not contained in the Ueno locus it has finite stabilizer and is not special. Thus, contradicting the Bogomolov conjecture over $L$.
\end{proof}

For tori we prove the Bogomolov conjecture using equidistribution in the same way as in \cite{bilu_equidistribution}.

\begin{proposition}\label{prop:archimedean_torus}
    The Bogomolov conjecture for $\bbG_m^t$ is satisfied over an archimedean GVF $K$.
\end{proposition}

\begin{proof}
    By the same argument as above we are reduced to the case that $K$ is countable and $X \subset \bbG^t_m$ contains a Zariski dense set of small points with respect to the Weil height. Assume wlog.\ that $X$ generates $\bbG_m^t$. Then, for any non-trivial character $\chi:\bbG_m^t \to \bbG_m$ the image $\chi(X) = \bbG_m$. Let $(x_i) \in X(K)$ be a small generic sequence. Then, the sequence $(\chi(x_i)) \in \bbG_m(K)$ is a generic sequence of small points. Therefore, both the $x_i$ and the $\chi(x_i)$ satisfy the conditions of Theorem \ref{thm:equidistribution}.

    We represent $K$ by an adelic curve and assume that all its archimedean places are normalized. Denote the set of archimedean places by $\Omega_\infty$ and denote the restricted measure from the adelic curve structure by $\nu$. We consider the measure $\delta_{x_i}$ on the adelic space $(\bbG_m^t)^{\an}_{\Omega_\infty}$ defined by sending a function $f:(\bbG_m^t)^{\an}_{\Omega_\infty} \to \bbR$ to the integral \[
    \int_{\Omega_\infty}f_\omega(x_i)\nu(d\omega).
    \]
    Denote by $\pi$ the projection onto $(\bbC^*)^t$. We first note that by Theorem \ref{thm:equidistribution} applied to $X$ the measures $\delta_{x_i}$ equidistribute to some measure $\mu$. The measure $\chi_*\mu$ has to be the measure given by fibrewise integration over the Haar measure on the unit circle due to Theorem \ref{thm:equidistribution} applied to $\bbG_m$ or rather its compactification $\bbP^1$. By Fourier theory, it follows that for every measurable set $U \subset \Omega_\infty$ we obtain that $\pi_* \mu$ is a multiple of the Haar measure on $(S^1)^t$. Hence $\mu$ is given by the product measure of the Haar measure on $(S^1)^t$ and the measure on $\Omega_\infty$. However, then $\mu(X_{\Omega_\infty}) = 0$ contradicting the fact that $\mu$ is supported on the adelic space of $X$.
\end{proof}

\subsection{\label{sec:geometric}Geometric Bogomolov for traceless abelian varieties}

Here we prove the Bogomolov conjecture for traceless abelian varieties over globally valued fields of characteristic $0$ by using Gao's theorem on degenerate subvarieties.

\begin{lemma}\label{lemm:traceless_abelian}
    $\mathrm{(BC)}$ is true for every traceless abelian variety over a non-archimedean GVF of characteristic $0$.
\end{lemma}

\begin{proof}
    By Lemma \ref{lemm:finite_stabilizer} and Lemma \ref{lemm:generates_abelian_variety} we place ourselves in the setting that $X$ has finite stabilizer and generates $A$.
    
    We will use the height formulation of the Bogomolov conjecture. To be precise, we want to prove that if $A$ is a traceless abelian variety over a non-archimedean algebraically closed GVF $K$ of characteristic $0$ and $X$ is a subvariety satisfying $\NTht(X) = 0$, then $X$ is the translate of an abelian subvariety by a height $0$ point.

    Before proceeding with the proof we recall some basic facts on non-degeneracy of subvarieties of abelian schemes. Let $\sM$ be an ample symmetric geometric adelic line bundle on an abelian scheme $\sA \to S$ over $\bbC$. By \cite[\S 6.2.2]{yuan_zhang_adeliclinebundlesquasiprojective}, for a subscheme $\sX$ of an abelian scheme $\sA$ it is equivalent for the Betti map to generically have full rank and for $\sM|_\sX$ to be big. A variety $\sX$ satisfying these equivalent conditions is called \emph{non-degenerate}. By \cite[Theorem 4.1]{yuan_uniform_bogomolov_curves}, the line bundle $\langle \sM, \dots, \sM \rangle_{\sX/S}$ is big if and only if $\sM^{\boxtimes N}$ is a big line bundle on $\sX^{[N]} = \sX \times_S \dots \times_S \sX$ for sufficiently large $N$. A variety $\sX$ satisfying these equivalent conditions is called \emph{potentially non-degenerate}.

    Since the height is stable under base change we may assume that it suffices to prove the Bogomolov conjecture for pairs $X \subset A$ defined over a GVF $K$ of finite type over the constant field $k$ and that $K$ is of minimal transcendence degree among the fields of definition. We spread out $X \subset A$ to families $\sX \subset \sA$ over a variety $S$ with fraction field $K$. Let $\sM$ be an ample symmetric geometric adelic line bundle on $\sA$. By Lemma \ref{lemm:big_line_bundle_triviality_of_GVF}, it suffices to prove that the Deligne pairing $\langle \sM, \dots, \sM\rangle_{\sX/S}$ on $S$ is a big line bundle. We may assume that the moduli map from $S$ to the relative Hilbert scheme is generically finite and thereby reduce to the case that $S$ is a closed irreducible subvariety of the relative Hilbert scheme $\resHilbert$. Let $\ov{\eta}$ denote a geometric generic point of $S$. The assumptions made on $X\subset A$ imply that up to passing to an open subvariety of $S$, $\sX$ satisfies the conditions
    \begin{enumerate}
        \item $\sX_{\ov{\eta}}$ is an irreducible subvariety of $\sA_{\ov{\eta}}$
        \item $\sX_s$ generates $\sA_s$ for all $s \in S(\bbC)$
        \item the subvariety $\sX_{\ov{\eta}}$ has finite stabilizer.
    \end{enumerate}
    By \cite[Lemma 3.4]{gao_ge_kuehne_uniform_mordell_lang}, $\sX$ is potentially non-degenerate finishing the proof.
\end{proof}

\subsection{\label{sec:quasi-split}Bogomolov for quasi-split semiabelian varieties}

In this section, we reduce the Bogomolov conjecture for quasi-split semiabelian varieties to the case of traceless abelian varieties. Yamaki originally proved this reduction in the case of polarized function fields for abelian varieties, cf.\ \cite{yamaki_trace_of_abelian_varieties_geometric_bogomolov}. We extend this to globally valued fields and quasi-split semiabelian varieties.

We begin by proving a Bogomolov conjecture for isotrivial varieties. Let $k$ be an algebraically closed field endowed with the trivial GVF structure. Let $X$ be a variety over $k$. Let $L$ be an ample line bundle on $X$. Note that it defines a geometric adelic line bundle on $X$ and it therefore makes sense to consider the height with respect to $L$. Let $K$ be a finitely generated GVF extension of $k$ with constant field $k$. We say that $Z\subset X_K$ is small if $h_L(Z) = 0$ or equivalently $\essmin(L|_Z) = 0$.

\begin{proposition}[Isotrivial Bogomolov conjecture] \label{prop:isotrivial_bogomolov}
    Let $Z\subset X_K$ be a small subvariety. Then $Z$ can be defined over $k$.
\end{proposition}
\begin{proof}
    Let $F$ be a GVF extension of $K$ with constant field $F_{\textrm{const}}$ such that $Z$ contains a Zariski dense set of $F_{\textrm{const}}$-points. Such an $F$ exists by Proposition \ref{prop:lower_zhang_ineq}. Then, $Z$ is defined over $F_{\textrm{const}}$. Since it is also defined over $K$ it is defined over $K\cap F_{\textrm{const}} =k$ by \cite[IV$_2$ 4.8.11]{EGA}.
\end{proof}

\begin{lemma}
    Let $K$ be an algebraically closed geometric GVF. Then, there is a unique GVF structure on $K(x)$ such that $\height(x) = 0$.\footnote{This result holds over archimedean GVFs as well. It is a generalization of Bilu's equidistribution theorem due to Ben Yaacov and Hrushovski and can be found in unpublished notes.}
\end{lemma}

\begin{proof}
    Without loss of generality assume that $K$ is countable. Let $(K,(\Omega,\sA,\nu),\phi)$ be a representing adelic curve for $K$ with no measure on the trivial absolute value. By the strong triangle inequality $\height(x+a) = \height(a)$ for all $a \in K$. It follows that any adelic curve representing $K(x)$ is supported on absolute values such that $|x+a|= 1$ if $|a| \leq 1$. For any non-trivial non-archimedean place $v$ of $K$, there is a unique extension to an absolute value $K(x)$ with this property, the Gauss norm. Hence, we may find a representing adelic curve for $K(x)$ of the form $(K(x),(\bbP^1(K)) \sqcup \Omega,\sB,\mu),\psi)$, where on $\Omega$ the map $\psi$ maps to the Gauss point over the corresponding non-trivial valuation of $K$. The measure $\mu$ restricted to $\Omega$ is determined by $\nu$. The places extending the trivial absolute value on $K$ can be identified with $\bbP^1(K)$. The measure on each such $a \in \bbP^1(K))$ is determined by the properties that $\mu(\{\infty\}) = 0$ and the product formula for $x-a$ for all $a \in K$.
\end{proof}

\begin{lemma}\label{lemm:constant_field_canonical_GVF}
    Let $A$ be a traceless abelian variety over a geometric GVF $K$ and assume that the Bogomolov conjecture holds for $A$. Then, the algebraic closure of its function field with its canonical GVF structure $\ov{K(\eta)}$ has constant field $k$.
\end{lemma}

\begin{proof}
    If $A = A_1 \times A_2$ the base change $A_1 \otimes \ov{K(A_2)}$ has trivial $\ov{K(A_2)}/\ov{k(A_2)}$-trace by \cite[Lemma A.1]{yamaki_trace_of_abelian_varieties_geometric_bogomolov}. In particular, if the Lemma is proven for $A_2$ over $K$ and $A_1 \otimes \ov{K(A_2)}$ over $\ov{K(A_2)}$ then we can conclude the the lemma for $A$. By Poincar\'e reducibility we may assume that $A$ is simple. By Lemma \ref{lemm:constant_field_symmetric_extension} it suffices to study the constant field of $K(\eta)$.

    Let $s \in K(\eta) \setminus K$. In order for $s$ to have height $0$, the vanishing locus $\textrm{div}(s)$ has to be of height $0$, hence a height $0$ translate of an abelian subvariety. We are hence reduced to the case that $A$ is an elliptic curve $E$. Observe that $E$ can be defined over $k(j(E))$ and let $v$ denote the $j(E)^{-1}$-adic valuation $v$ on $k(j(E))$. Represent $K$ by an adelic curve $(K,(\Omega,\sA,\nu),\phi)$ such that the places $\Omega_{\textrm{bad}}$ of bad reduction extend $v$. Let $\ov{\sM}$ denote the canonically metrized symmetric ample line bundle associated to the divisor $[0]$.
    
    Let $K(\eta)$ be the function field of $E$ and let $x \in K(\eta)\setminus k$ of height $0$. Denote the GVF functional of $K(\eta)$ by $\phi$. Then, $K(\eta)/K(x)$ is finite and $K(x)$ is endowed with the unique GVF structure extending $K$ satisfying $\height(x) = 0$. For every place $\omega$ of $K$ the GVF measure for $K(x)$ is supported on the Gauss point. We define an adelic divisor $\ov{D}$ on $K(x)$ by pullback from $k(j(E))(x)$ to be a trivial divisor together with a continuous non-negative function $f$ on $\bbP^1_v$ that vanishes precisely on the Gauss point. We notice that the evaluation of the GVF functional of $K(x)$ at $\ov{D}$ has to be $0$. We consider the pullback of $\ov{D}$ to $K(E)$. Denote it too by $\ov{D}$ and its Green's function at each place by $f_\omega$.
    
    We see that $\int_{E_\omega^{\an}} f_\omega c_1(\ov{\sM}_\omega) > 0$ for all $\omega \in \Omega_{\textrm{bad}}$, where by \cite[Corollary 7.3]{gubler_non-archimedean_canonical_measures_abelian_varieties} the $c_1(\ov{\sM}_\omega)$ is the Haar measure on the skeleton of $E_\omega^{\an}$. In particular, \[\phi(\ov{D}) = \int_{\Omega_{\textrm{bad}}} \int_{E_\omega^{\an}} f_\omega c_1(\ov{\sM}_\omega) \nu(d\omega) > 0.\] This contradicts the assumption that $K(\eta)$ is a GVF extension of $K(x)$.
\end{proof}

We also require an archimedean analogue of this result.

\begin{lemma}\label{lemm:archimedean_height_0_divisor}
    Let $K$ be a countable archimedean GVF represented by an adelic curve $(K,(\Omega,\sA,\nu),\phi)$. Suppose $K(X)$ is a GVF extension of $K$ to the function field of a variety $X$ with a representing adelic curve $(K(X),(\Psi,\sB,\mu),\psi)$ containing a measure subspace of the form $d\mu_\omega d_\omega$ on $X^{\an}_U\setminus \bigcup_{Z \subset X \textrm{ cld}}Z_U^{\an}$ such that $\Supp(\mu_\omega) = X_\omega^{\an}$ for a set $U\subset \Omega$ of positive measure. Denote the GVF functional for $\ov{K(X)}^{\sym}$ by $\phi$. Then, for any effective adelic divisor $\ov{\sD}$ on $\ov{K(X)}$ such that $\phi(\ov{\sD}) = 0$ it follows that $\ov{\sD}$ is geometrically trivial. In particular, the set $\{x \in \ov{K(\eta)}\ |\ \height(x) = 0\}$ of elements of height $0$ is contained in $K$.
\end{lemma}

\begin{proof}
    By taking norm line bundles it suffices to prove the result in case $\ov{\sD}$ is defined on a birational modification of $X$. The in particular part of the statement follows from considering $\ov{\sD} = \adiv(s)\wedge 0$ for $s \in K(X)\setminus K$.

    Let $\ov{\sD}$ be an effective adelic divisor that is not geometrically trivial. Then, the Green's function of $\ov{\sD}$ at any archimedean place $\omega$ is positive in a euclidean neighborhood of $\textrm{Supp}(\sD)$. Since the support of $\mu_\omega$ is $X_\omega^{\an}$ it follows that 
    \[
        \phi(\ov{\sD}) \geq \int_{\Omega_\infty}\int_{X_\omega^{\an}} g_{\ov{\sD}}(x) \mu_\omega(d x) \nu(d\omega) > 0
    \]
    implying the claim.
\end{proof}

\begin{corollary}\label{cor:height_0_in_canonical_GVF}
    Let $A$ be an abelian variety over an archimedean GVF $K$. Let $\ov{K(\eta)}$ denote the function field of $A$ with its canonical GVF structure. Then, the set $\{x \in \ov{K(\eta)}\ |\ \height(x) = 0\}$ of elements of height $0$ is contained in $K$.
\end{corollary}

\begin{proof}
    If $X=A$ is an abelian variety over $K$ then the canonical GVF structure on $K(\eta)$ satisfies the conditions of Lemma \ref{lemm:archimedean_height_0_divisor}. For every archimedean place $\omega$ one may take $\mu_\omega$ to be the Haar measure on $A_{\omega}^{\an}$.
\end{proof}

The following lemma is a generalization of \cite[Proposition 4.6]{luo2025geometricbogomolovconjecturesemiabelian} encompassing both the non-archimedean and the archimedean setting.

\begin{lemma}\label{lemm:product_isotrivial_abelian}
    Let $A$ be an abelian variety over a GVF $K$ with an ample symmetric canonically metrized line bundle $\ov{\sM}$. If $K$ is non-archimedean assume that $A$ has trivial $K/k$-trace. Suppose further that the Bogomolov conjecture is satisfied for $A$.
    \begin{enumerate}
        \item If $K$ is non-archimedean let $Y$ be a variety over $k$ and $M_Y$ be an ample line bundle on $Y$ which by pullback defines an adelic line bundle $\ov{M}_Y$ on $Y_K$. Let $X \subset Y_K \times A$ be of height $0$ wrt $\ov{M}_Y \boxtimes \ov{\sM}$. Then $X$ is of the form $W_K \times (A' + a)$, where $W\subset Y$ is a subvariety defined over $k$, $A'$ is an abelian subvariety of $A$ and $a \in A(K)$ is a point of height $0$.
        \item If $K$ is archimedean the Bogomolov conjecture is satisfied for $\bbG^t_m \times A$.
    \end{enumerate}
    In particular, if the Bogomolov conjecture is satisfied for $A$ and $G$ is an isotrivial semiabelian variety. Then, the Bogomolov conjecture is satisfied for $G \times A$.
\end{lemma}
    
\begin{proof}
    In the archimedean setting denote $Y=\bbG^t_m$. Let $X$ be a small subvariety of $Y_K \times A$. By the Bogomolov conjecture for abelian varieties we may assume that the projection $\pi(X)$ of $X$ onto $A$ is $A$ itself. Let $\eta$ denote the generic point of $A$ and denote by $K(\eta)$ the function field of $A$ with the canonical GVF structure. Then, the irreducible components of $X_{\ov{K(\eta)}}$ are small subvarieties.

    By Proposition \ref{prop:isotrivial_bogomolov} and Lemma \ref{lemm:constant_field_canonical_GVF} the claim follows in the non-archimedean case as every irreducible component of $X_{\ov{K(\eta)}}$ is defined over $k$.
    
    Let us proceed to the archimedean case. Let $Z$ be an irreducible component of $X_{\ov{K(\eta)}}$. By Proposition \ref{prop:archimedean_torus}, we see that $Z$ must be the translate of a subgroup by a point of height $0$. We may hence assume $Z$ is a point. Then, it defines a tuple of height $0$ in $(\ov{K(\eta)}^\times)^t$. By Corollary \ref{cor:height_0_in_canonical_GVF}, it follows that $Z$ is defined over $K$. This proves that $X$ is a translate of $T \times A$ for a subgroup $T$ by a height $0$ point, thus proving the Bogomolov conjecture.
\end{proof}

\subsection{\label{sec:reduction_to_quasisplit}Reduction to the quasi-split case}

In this section, we reduce the Bogomolov conjecture for general semiabelian varieties to the quasi-split case by proving the following lemma.

\begin{lemma}\label{lemm:reduction_to_quasisplit}
    Let $X \subset G$ be a small subvariety of a semiabelian variety with finite stabilizer generating $G$. Then, $G$ is isogenous to a quasi-split semiabelian variety.
\end{lemma}

Let $X$ be a small subvariety with finite stabilizer generating $G$. By taking an isogeny we may assume that the abelian quotient of $G$ is the product $A_{0,K} \times A_1$ of its trace with a traceless abelian variety.

An important tool we use is the relative Faltings-Zhang map as introduced in \cite{luo2025geometricbogomolovconjecturesemiabelian}. We define an isomorphism of semiabelian varieties exhibiting the fact that $G$ is $\bbG_m^t$-torsor over $A$ by $\beta_n: G^n_{/A} \to (\bbG_m^t)^{n-1} \times G$ by $(g_1, \dots, g_n) \mapsto (g_1 - g_2, g_2-g_3, \dots, g_n - g_{n-1},g_1)$. We define the relative Faltings-Zhang map $\alpha_n$ by composing $\beta_n$ with $\id_{(\bbG_m^t)^{n-1}} \times \pi$. It is named as such because it recovers the Faltings-Zhang map for $\bbG^t_m$ on fibers of $\pi$.
\[\begin{tikzcd}
     &G^n_{/A}:=\underbrace{G\times_A\cdots\times_A G}_{n\textrm{ times}} \arrow[r,"\alpha_n"]\arrow[rd,"\beta_n"] & \mathbb G_m^{t(n-1)}\times A\\
     & & \mathbb G_m^{t(n-1)}\times G\arrow[u,"\mathrm{id}\times \pi"].
    \end{tikzcd}
\]

The relative Faltings-Zhang map behaves similarly to the usual Faltings-Zhang map. 

\begin{lemma}[Proposition 5.8 \cite{luo2025geometricbogomolovconjecturesemiabelian}]\label{lemm:faltings_zhang_birational}
    If $\mathrm{Stab}(X)$ is finite, then for $n\gg 0$ there exists an irreducible component $Z_n \subset X^n_{/A}$ such that $Z_n\rightarrow \alpha_n(Z_n)$ is generically finite and such that $\alpha_n(Z_n) = \alpha_n(X^n_{/A})$.
\end{lemma}

\begin{lemma}[Proposition 5.7 \cite{luo2025geometricbogomolovconjecturesemiabelian}]\label{lemm:small_fibre_powers}
    Let $X'$ be an irreducible component of the $n$-fold fibre power $X^n_{/A}\subset G^n_{/A}$ of a small subvariety. Then $X' \subset G^n_{/A}$ is small.
\end{lemma}

\begin{proof}
    The proof relies merely on the projection formula which also holds for GVFs/adelic curves, c.f.\ \cite[Theorem 4.4.9]{Adelic_curves_2}.
\end{proof}

We now apply the relative Faltings-Zhang map to our setting. By Lemma \ref{lemm:faltings_zhang_birational}, we let $Z_n \subset X^n_{/A}$ be an irreducible component such that the map $Z_n \to \alpha_n(Z_n)$ is generically finite. Since taking fibre powers preserves smallness (Lemma \ref{lemm:small_fibre_powers}) and $\beta_n$ is an isomorphism, both $\alpha_n(Z_n)$ and $\beta_n(Z_n)$ are small subvarieties.

We focus now on $Z = \beta_n(Z_n)$. $Z \subset \bbG_m^N \times G$ is a small subvariety whose projection under $\pi: \bbG_m^N \times G \to \bbG_m^N \times A$ onto $\pi(Z)$ is generically finite. Our aim is to prove that $G$ up to isogeny is of the form $G = G_{0,K} \times A'$ for a traceless abelian variety $A'$.

If $K$ is non-archimedean, $\pi(Z)$ is of the form $W_K \times A_1$ where $W$ is an subvariety of $\bbG_m^N \times A_0$ defined over $k$ such that the projection onto $A_0$ generates $A_0$ by the Bogomolov conjecture for quasi-split semiabelian varieties. If $K$ is archimedean, we may similarly reduce to the case that $\pi(Z) = \bbG_m^N \times A$. Compactify $\bbG_m^N$ and let $\ov{M} = \arho(\Delta)$ be a canonically metrized line bundle on its compactification. Denote by $\ov{W}$ the closure of $W$ in this compactification. Then, the adelic line bundle $\ov{M}_{\ov{W}}$ on $\ov{W}_K$ obtained by pulling back an ample line bundle on $\ov{W}$ defined over $k$. This puts us precisely in the situation of \cite[Theorem 5.6]{luo2025geometricbogomolovconjecturesemiabelian} except that we work over GVFs.

We note that we can interpret $G$ as a $\bbG_m^t$-bundle over $A$. As such $Z$ gives rise to a \emph{meromorphic multisection} of this $\bbG_m^t$-bundle over $\ov{W}_K \times A_1$. By this we mean a meromorphic section defined over a generically finite cover $Y$. Equivalently, we obtain $t$ meromorphic multisections $s_1, \dots, s_t$ of $\bbG_m$-bundles. The height of $Z$ agrees with the height of the adelic divisor $|\textrm{div}(s_1)|+ \dots + |\textrm{div}(s_t)|$ with respect to the canonical GVF structure on $\ov{K(W_K \times A_1)}$.

\begin{proof}[Proof of Lemma \ref{lemm:reduction_to_quasisplit}]
    We prove that $G$ is quasi-split after the reductions performed previously. By Lemma \ref{lemm:line_bundle_isotrivial_times_abelian}, the line bundles defined by $G$ over $W_K \times A'$ is of the form $Q_K \boxtimes \sT$, where $Q$ is a line bundle defined over $k$ and $\sT$ is torsion. We note that $Q$ extends uniquely to a line bundle on $A_0$. The projection $W \to A_0$ factors over the Albanese variety $\Alb(W)$. Since the image generates $A_0$, the morphism $\Alb(W) \to A_0$ is surjective. Numerically trivial line bundles extend uniquely to the Albanese variety. By assumption, the line bundle is defined by an element $\Alb(W)^\vee(k)$ in the image of $A_0^\vee(K)$, but since $A_0^\vee \to \Alb(W)^\vee$ is injective it follows that the line bundle over $A_0$ can be defined over $k$.
\end{proof}

The rest of the section is devoted to proving Lemma \ref{lemm:line_bundle_isotrivial_times_abelian}. We will need the following lemma on numerically trivial line bundles. It is not original, but we can not find a reference.

\begin{lemma}\label{lemm:num_triv_on_product}
     Let $X$ and $Y$ be two proper varieties over an algebraically closed field $K$. Then, we have an isomorphism $\Pic_0(X) \times \Pic_0(Y) \cong \Pic_0(X \times Y)$. In other words, any numerically trivial line bundle on $X\times Y$ is of the form $Q_1 \boxtimes Q_2$.
\end{lemma}

\begin{proof}
    Fix two base points $x \in X(K)$ and $y \in Y(K)$. Let $L$ be the universal line bundle on $X \times Y \times \Pic_0(X \times Y)$. Let $L_x$ be its restriction to $\{x\}\times Y \times \Pic_0(X \times Y) \cong Y \times \Pic_0(X \times Y)$ and $L_y$ be its restriction to $X \times \{y\} \times \Pic_0(X \times Y)\cong X \times \Pic_0(X \times Y)$. Denote the projections by $\pi_x$ and $\pi_y$ respectively. Then, we claim that $L \cong \pi_x^* L_x \otimes \pi_y^* L_y$. By the theorem of the cube, see \cite[0BF4]{stacks-project}, it suffices to check that this is true after restricting for $\{x\}\times Y \times \Pic_0(X \times Y)$, $X \times \{y\} \times \Pic_0(X \times Y)$ and $X \times Y \times \{e\}$. This is true by construction.
\end{proof}

\begin{lemma}\label{lemm:line_bundle_isotrivial_times_abelian}
    Let $\ov{\sQ}$ be a numerically trivial line bundle on $W_K \times A'$ with its canonical metric and $A'$ is a traceless abelian variety. Let $s$ be a meromorphic multisection of $\ov{\sQ}$ of height $0$ with respect to the canonical GVF structure on $K(W_K \times A')$. Then, the line bundle $\ov{\sQ}$ is of the form $Q_K \boxtimes \sT$, where $Q$ is a line bundle defined over $k$ and $\sT$ is a torsion line bundle.
\end{lemma}

\begin{proof}
    Let $Y$ denote $W_K \times A'$. It suffices to consider the case of a section as the norm line bundle of $p^*\ov{\sQ}$ is $[K(X):K(Y)] \cdot \ov{\sQ}$ for any generically finite map $p:X \to Y$.
    
    By Lemma \ref{lemm:num_triv_on_product}, the line bundle is of the form $Q_1 \boxtimes Q_2$. After pulling back to the function field of $W$ we apply Lemma \ref{lemm:sections_for_traceless} to show that $Q_2$ is torsion. On the other hand, pullback $Q_1$ to $X_{K(A)}$. Then, by isotrivial Bogomolov $s$ can only vanish on subvarieties defined over the constant field of $K(A)$, which by Lemma \ref{lemm:constant_field_canonical_GVF} is $k$. As such, $Q_1$ is defined over $k$ as required.
\end{proof}

\begin{lemma}\label{lemm:local_contribution_elliptic_curves}
    Let $E$ be an elliptic curve of multiplicative reduction over a valued field $K$. Let $D = P_1 + \dots +P_n - P'_1 - \dots - P'_n$ be a numerically trivial divisor. Suppose that the canonical Green's function $g_D$ is constant on the skeleton of $E^{\an}$ and let $r$ denote the retraction to the skeleton. Then, $r(P_1), \dots, r(P_n)$ and $r(P'_1), \dots, r(P'_n)$ agree with multiplicities. In particular, the point $Q \in E(K)$ defined by $D$ satisfies $r(Q) = 0$.
\end{lemma}

\begin{proof}
    We observe by Lemma \ref{lemm:numerically_trivial_are_differences} that $g_D(x)$ up to an additive constant is of the form $g_{[0]}(x-P_1) + \dots + g_{[0]}(x-P_n) - g_{[0]}(x-P'_1) - \dots - g_{[0]}(x-P'_n)$, where $g_{[0]}$ denotes the Green's function for the divisor $[0]$ with the canonical metric. The Green's function is understood explicitly in \cite[\S 8.5]{deJong_shokrieh_canonical_local_heights} following work of Tate. 

    Identifying the skeleton of $E^{\an}$ with $\bbR / l \bbZ$ for a suitable $l \in \bbR$ the restriction $g$ of $g_{[0]}$ to the skeleton is
    \[
        \frac{l}{2} B_2(\nu/l) - \frac{1}{12} l,
    \]
    where $B_2$ denotes the second Bernoulli polynomial $B_2(T) = T^2- T + 1/6$. We note that $g$ is differentiable at every non-zero point. Its derivative is given by $\frac{2\frac{\nu}{l} - 1}{2}$. At $0$ the right side derivative is $-\frac{1}{2}$, but its left side derivative is $\frac{1}{2}$. In particular the restriction of $g_D$ to the skeleton which is
    \[
        g(x+r(P_1)) + \dots + g(x+r(P_n)) - g(x+r(P'_1)) - \dots - g(x+r(P'_n))
    \]
    can only be differentiable if $r(P_1), \dots, r(P_n)$ and $r(P'_1), \dots, r(P'_n)$ agree with multiplicities.
\end{proof}

\begin{lemma}\label{lemm:num_triv_height_0_elliptic_curve}
    Let $E$ be a non-isotrivial elliptic curve over a geometric GVF $K$. Let $\phi$ denote the GVF functional for the canonical GVF structure on $K(E)$. Let $\ov{\sD}$ be an adelic divisor such that $\sO(\ov{\sD})$ is a canonically metrized numerically trivial line bundle and $\phi(|\ov{\sD}|) = 0$. Then, $\sO(\ov{\sD})$ is a torsion line bundle.
\end{lemma}

\begin{proof}
    Let $\ov{\sD}$ be an adelic divisor such that $\sO(\ov{\sD})$ is a canonically metrized numerically trivial line bundle and $\phi(|\ov{\sD}|) = 0$. Replace $K$ by a countable subfield over which $E$ and $\sD$ are defined. Let $Q \in E(K)$ be the point unique point such that $\sD$ is rationally equivalent to $[Q] - [0]$.
    
    Observe that $E$ can be defined over $k(j(E))$ and let $v$ denote the $j(E)^{-1}$-adic valuation $v$ on $k(j(E))$. Represent $K$ by an adelic curve $(K,(\Omega,\sA,\nu),\phi)$ such that the places $\Omega_{\textrm{bad}}$ of bad reduction extend $v$. Denote the Green's function of $\ov{\sD}$ by $g_{\ov{\sD}}$. Let $\ov{\sM}$ denote the canonically metrized symmetric ample line bundle associated to the divisor $[0]$. By following the explicit description of the polarized adelic curve in \S \ref{sec:polarized_GVF} we compute
    \begin{align*}
        \phi(|\ov{\sD}|) = \sum_{P \in \operatorname{PDiv}(E)} |\ord_P(\sD)| \NTht(P) + \int_{\Omega}\int_{E^{\an}_\omega}|g_{\ov{\sD}}| c_1(\ov{\sM}_\omega)\nu(d \omega). 
    \end{align*}
    Since the integrand is nonnegative it follows that $\NTht(P) = 0$ at all places where $\sD$ has a zero or pole. In particular, the point $Q \in E(K)$ defined by $\sD$ satisfies $\NTht(Q) = 0$. On the other hand, we conclude that
    \begin{align*}
        \int_{\Omega_{\textrm{bad}}}\int_{E^{\an}_\omega}|g_{\ov{\sD}}| c_1(\ov{\sM}_\omega)\nu(d \omega). 
    \end{align*}
    Since $c_1(\ov{\sM}_\omega)$ is the Haar measure on the skeleton of $E^{\an}_\omega$ by \cite[Corollary 7.3]{gubler_non-archimedean_canonical_measures_abelian_varieties} we deduce that for almost all $w \in \Omega_{\textrm{bad}}$, $|g_{\ov{\sD}}|$ is constantly $0$ on the skeleton of $E^{\an}_w$. Let $r_w$ denote the retraction to the skeleton at $w$. Then, it follows from Lemma \ref{lemm:local_contribution_elliptic_curves} that $r_w(Q) = 0$ at almost all places $w$ of bad reduction.

    Assume that $Q$ is not torsion and of height $0$. Then, the sequence of $[n]Q$ is generic of height $0$ and hence equidistributes to the Haar measure on the skeleton by Theorem \ref{thm:equidistribution}. This contradicts the fact that the retraction of $Q$ to the skeleton is $0$ at almost all places of bad reduction. It follows that $Q$ is torsion.
\end{proof}

\begin{lemma}\label{lemm:sections_for_traceless}
    Let $Q$ be a numerically trivial line bundle over an abelian variety $A$ with a meromorphic multisection $s$. Suppose that the height of the section $s$ wrt the canonical GVF structure on $\ov{K(A)}$ is $0$. If $A$ is of trivial trace, then $Q$ is torsion.
\end{lemma}

\begin{proof}
    By the same argument as in Lemma \ref{lemm:line_bundle_isotrivial_times_abelian} it suffices to restrict to the case of a meromorphic section. By Poincar\'e reducibility we may without loss of generality assume that $A$ is of the form $A_1 \times \dots\times A_k$. By Lemma \ref{lemm:num_triv_on_product} we have $Q = Q_1 \boxtimes \dots \boxtimes Q_k$. It suffices to prove that $Q_1$ is torsion. For this we consider the base change of $Q$ and $s$ to $A_1 \otimes \ov{K(A_2 \times \dots\times A_k)}$. This abelian variety has trivial $\ov{K(A_2 \times \dots\times A_k)}/\ov{k(A_2 \times \dots\times A_k)}$-trace by \cite[Lemma A.1]{yamaki_trace_of_abelian_varieties_geometric_bogomolov} and in particular is traceless over $\ov{K(A_2 \times \dots\times A_k)}$ with its canonical GVF structure. Since the canonical GVF structure on $\ov{K(A_1 \times A_2 \times \dots\times A_k)}$ can be obtained as the canonical GVF for $A_1 \otimes \ov{K(A_2 \times \dots\times A_k)}$, where $\ov{K(A_2 \times \dots\times A_k)}$ is endowed with its canonical GVF structure. This reduces us to prove the lemma in the case that $A$ is simple.
    
    If $Q$ is not trivial, $s$ has to vanish on a subvariety of $A$. This has to have height $0$. By the Bogomolov conjecture it is the translate of an abelian subvariety. This cannot happen if $\dim A > 1$ as $A$ was assumed simple. Hence, we are reduced to the case that $A = E$ is an elliptic curve with $j$-invariant in $K\setminus k$.
    
    If $K$ is archimedean it follows from Lemma \ref{lemm:archimedean_height_0_divisor} that $s$ has to be a constant section and $Q$ is trivial. If $K$ is non-archimedean Lemma \ref{lemm:num_triv_height_0_elliptic_curve} applies to show that $Q$ is torsion.
\end{proof}

\subsection{Proof of main result}

\bogomolovabvarcharzero*

\begin{proof}
    This follows from the traceless case in Lemma \ref{lemm:traceless_abelian} as a special case of Lemma \ref{lemm:product_isotrivial_abelian}.
\end{proof}

\bcsa*

\begin{proof}
    The latter assertions follow immediately from the first and Theorem \ref{thm:bogomolov_characteristic_0}. The characteristic $p$ part follows since the Bogomolov conjecture for elliptic curves is trivial.

    It remains to prove the first assertion. Let $X \subset G$ be a small subvariety. We need to prove that $X$ is special. We reduce to the case that $X$ has a finite stabilizer and generates $G$ by Lemma \ref{lemm:finite_stabilizer} and Lemma \ref{lemm:generates_abelian_variety}. By Lemma \ref{lemm:reduction_to_quasisplit} it suffices to consider the case that $G$ is quasi-split. The Bogomolov conjecture for quasi-split semiabelian varieties is proven in Lemma \ref{lemm:product_isotrivial_abelian}.
\end{proof}

\printbibliography

\end{document}